\documentclass[graybox]{svmult}

\usepackage{mathptmx}       
\usepackage{helvet}         
\usepackage{courier}        
\usepackage{type1cm}        
\usepackage{makeidx}         
\usepackage{graphicx}        
\usepackage{multicol}        
\usepackage[bottom]{footmisc}

\newcommand{\ice}{{\rm ICE}}
\newcommand{\acehat}{\mbox{$\widehat {\rm ACE}$}}
\newcommand{\muhat}{\mbox{$\widehat {\rm \mu}$}}

\newcommand{\sce}{{\rm SCE}}
\newcommand{\idle}{\emptyset}

\newcommand{\aipw}{{\rm AIPW}}
\renewcommand{\Pr}{{\rm P}}
\newcommand{\as}{\mbox{ \rm a.s.\ }}
\newcommand{\by}{{\rm by}}

\newcommand{\ld}{{\rm LD}}
\newcommand{\lp}{{\rm LP}}
\newcommand{\qd}{{\rm QD}}

\newcommand{\pv}{{\rm PV}}
\newcommand{\epv}{{\rm EPV}}
\newcommand{\ps}{{\rm PS}}
\newcommand{\var}{{\rm Var}}
\newcommand{\asyvar}{{\rm Var_{\emph{.asy}}}}

\newcommand{\logit}{{\rm logit}}

\newcommand{\dt}{{\rm DT}}

\newcommand{\HIDE}[1]{ }
\newcommand{\COMMENT}[1]{ }

\newcommand{\half}{\frac{1}{2}}

\newcommand{\secref}[1]{\mbox{\S$\,$\ref{sec:#1}}}

\newcommand{\lemref}[1]{\mbox{Lemma~\ref{lem:#1}}}
\newcommand{\thmref}[1]{\mbox{Theorem~\ref{thm:#1}}}
\newcommand{\corref}[1]{\mbox{Corollary~\ref{cor:#1}}}

\newcommand{\transp}{^{\prime}}

\newcommand{\cov}{{\rm Cov}}

\newcommand{\normal}{{\cal N}}

\newcommand{\cip}{\mbox{$\perp\!\!\!\perp$}}

\newtheorem{expl}{Example}[section]
\newtheorem{definer}{Definition}[section]

\newtheorem{thm}[theorem]{Theorem}
\newtheorem{algor}{Algorithm}[section]

\newtheorem{rem*}{Remark}[section]

\newcommand{\halm}{\hspace*{\fill} $\Box$\par}

\newcommand{\at}{\mbox{$\bm \alpha_3$}}

\newcommand{\eg}{e.g.}
\newcommand{\ie}{i.e.}

\newcommand{\ace}{\mbox{ACE}}
\newcommand{\e}{\mbox{E}}

\makeindex             

\begin{document}

\title*{Sufficient Covariate, Propensity Variable and Doubly Robust Estimation} 
\author{Hui Guo, Philip Dawid and Giovanni Berzuini}
\institute{Hui Guo \at Centre for Biostatistics, Institute of
  Population Health, The University of Manchester, Jean McFarlane
  Building, Oxford Road, Manchester M13 9PL, UK,
  \email{hui.guo@manchester.ac.uk} 
  \and Philip Dawid \at Statistical Laboratory, University of Cambridge, Wilberforce Road, Cambridge CB3
  0WB, UK, \email{apd25@cam.ac.uk} 
  \and Giovanni Berzuini \at Department of Brain and Behavioural Sciences, University of Pavia, Pavia, Italy, \email{giomanuel\_b@hotmail.com}}%

\maketitle

\abstract*{Each chapter should be preceded by an abstract (10--15
  lines long) that summarizes the content. The abstract will appear
  \textit{online} at \url{www.SpringerLink.com} and be available with
  unrestricted access. This allows unregistered users to read the
  abstract as a teaser for the complete chapter. As a general rule the
  abstracts will not appear in the printed version of your book unless
  it is the style of your particular book or that of the series to
  which your book belongs.  Please use the 'starred' version of the
  new Springer \texttt{abstract} command for typesetting the text of
  the online abstracts (cf. source file of this chapter template
  \texttt{abstract}) and include them with the source files of your
  manuscript. Use the plain \texttt{abstract} command if the abstract
  is also to appear in the printed version of the book.}

\abstract{Statistical causal inference from observational studies
  often requires adjustment for a possibly multi-dimensional variable,
  where dimension reduction is crucial. The propensity score, first
  introduced by Rosenbaum and Rubin, is a popular approach to such
  reduction.  We address causal inference within Dawid's
  decision-theoretic framework, where it is essential to pay attention
  to sufficient covariates and their properties. We examine the role
  of a propensity variable in a normal linear model. We investigate
  both population-based and sample-based linear regressions, with
  adjustments for a multivariate covariate and for a propensity
  variable. In addition, we study the augmented inverse probability
  weighted estimator, involving a combination of a response model and
  a propensity model.  In a linear regression with homoscedasticity, a
  propensity variable is proved to provide the same estimated causal
  effect as multivariate adjustment.  An estimated propensity variable
  may, but need not, yield better precision than the true propensity
  variable. The augmented inverse probability weighted estimator is
  doubly robust and can improve precision if the propensity model is
  correctly specified.}

\section{Introduction}
\label{sec:intro}
Causal effects can be identified from well-designed experiments, such
as randomised controlled trials (RCT), because treatment assignment is
entirely unrelated to subjects' characteristics, both observed and
unobserved. Suppose there are two treatment arms in an RCT: treatment
group and control group. Then the average causal effect (ACE) can
simply be estimated as the outcome difference of the two groups from
the observed data. However, randomised experiments, although ideal and
to be conducted whenever possible, are not always feasible. For
instance, to investigate whether smoking causes lung cancer, we cannot
randomly force a group of subjects to take cigarettes. Moreover, it
may take years or longer for development of this disease. Instead, a
retrospective case-control study may have to be considered. The task
of drawing causal conclusion, however, becomes problematic since
similarity of subjects from the two groups will rarely hold, \eg,
lifestyles of smokers might be different from those of
non-smokers. Thus, we are unable to ``compare like with like" -- the
classic problem of confounding in observational studies, which may
require adjusting for a suitable set of variables (such as age, sex,
health status, diet).  Otherwise, the relationship between treatment
and response will be distorted, and lead to biased inferences.  In
general, linear regressions, matching or subclassification are used
for adjustment purpose. If there are multiple confounders, especially
for matching and subclassification, identifying two individuals with
very similar values of all confounders simultaneously would be
cumbersome or impossible. Thus, it would be sensible to replace all
the confounders by a scalar variable. The propensity score \cite{pr:83} is
a popular dimension reduction approach in a variety of research
fields.

\section{Framework}
\label{sec:fw}
The aim of statistical causal inference is to understand and estimate
a ``causal effect'', and to identify scientific and in principle
testable conditions under which the causal effect can be identified
from observational studies. The philosophical nature of ``causality''
is reflected in the diversity of its statistical formalisations, as
exemplified by three frameworks:
\begin{enumerate}
\item{Rubin's potential response framework \cite{dr:74, dr:77, dr:78}
    (also known as Rubin's causal model) based on counterfactual
    theory;}
\item{Pearl's causal framework \cite{jp:95, jp:00} richly developed
    from graphical models;}
\item{Dawid's decision-theoretic framework \cite{apd:00, apd:02} based
    on decision theory and probabilistic conditional independence.}
\end{enumerate}
In Dawid's framework, causal relations are modelled entirely by
conditional probability distributions.  We adopt it throughout this
chapter to address causal inference; the assumptions required are, at
least in principle, testable.

Let $X$, $T$ and $Y$ denote, respectively, a (typically multivariate)
confounder, treatment, and response (or outcome). For simplicity, $Y$
is a scalar and $X$ a multi-dimensional variable. We assume that $T$
is binary: 1 (treatment arm) and 0 (control arm). Within Dawid's
framework, a non-stochastic regime indicator variable $F_T$, taking
values $\emptyset$, $0$ and $1$, is introduced to denote the treatment
assignment mechanism operating.  This divides the world into three
distinct regimes, as follows:

\begin{enumerate}
\item $F_T = \emptyset$: the observational (idle) regime. In this
  regime, the value of the treatment is passively observed and
  treatment assignment is determined by Nature.
\item $F_T = 1$: the interventional treatment regime, \ie, treatment
  $T$ is set to $1$ by manipulation.
\item $F_T = 0$: the interventional control regime, \ie, treatment $T$
  is set to $0$ by manipulation.
\end{enumerate}
For example, in an observational study of custodial sanctions, our
interest is in the effect of custodial sanction, as compared to
probation (noncustodial sanction), on the probability of
re-offence. Then $F_T=\emptyset$ denotes the actual observational
regime under which data were collected; $F_T = 1$ is the
(hypothetical) interventional regime that always imposes imprisonment;
and $F_T = 0$ is the (hypothetical) interventional regime that always
imposes probation.  Throughout, we assume full compliance and no
dropouts, \ie, each individual actually takes whichever treatment they
are assigned to. Then we have a joint distribution $P_f$ of all
relevant variables in each regime $F_T =f$ ($f=0, 1, \emptyset$).

In the decision-theoretic framework, causal assumptions are construed
as assertions that certain marginal or conditional distributions are
common to all regimes.  Such assumptions can be formally expressed as
properties of conditional independence, where this is extended to
allow non-stochastic variables such as $F_T$ \cite{apd:79a, apd:80,
  apd:02}. For example, the ``ignorable treatment assignment"
assumption in Rubin's causal model (RCM) \cite{pr:83} can be expressed
as

\begin{equation}
  \label{eq:ita}
  Y \cip F_T | T,
\end{equation}
read as ``$Y$ is independent of $F_T$ given $T$". However, this
condition will be most likely inappropriate in observational studies
where randomisation is absent.

Causal effect is defined as the response difference by manipulating
treatment, which purely involves interventional regimes. In
particular, the population-based average causal effect (ACE) of the
treatment is defined as:
\begin{equation}
  \label{eq:dtace1}
  \ace : = \e(Y | F_T = 1) - \e(Y | F_T = 0),
\end{equation}
or alternatively,
\begin{equation}
  \label{eq:dtace2}
  \ace : = \e_1(Y) - \e_0(Y) \footnote{For convenience, the values of the regime indicator $F_T$ are presented as subscripts.}.
\end{equation}

Without further assumptions, by its definition ACE is not identifiable
from the observational regime.

\section{Identification of ACE}
\label{sec:idace}

Suppose the joint distribution of ($F_T$, $T$, $Y$ ) is known and
satisfies (\ref{eq:ita}). Is ACE identifiable from data collected in
the observational regime? Note that (\ref{eq:ita}) demonstrates that
the distribution of $Y$ given $T = t$ is the same, whether $t$ is
observed in the observational regime $F_T = \emptyset$, or in the
interventional regime $F_T = t$. As discussed, this assumption would
not be satisfied in observational studies, and thus, direct comparison
of response from the two treatment groups cannot be interpreted as the
causal effect from observational data.

\begin{definition}
  The ``face-value average causal effect" (FACE) is defined as:
\end{definition}
\begin{equation}
  \label{eq:face}
  \mbox{FACE}:=\e_\emptyset(Y |T =1) - \e_\emptyset(Y |T =0). 
\end{equation}
It would be hardly true that FACE = ACE, as we would not expect the
conditional distribution of $Y$ given $T = t$ is the same in any
regime. In fact, identification of ACE from observational studies
requires, on one hand, adjusting for confounders, on the other hand,
interplay of distributional information between different regimes. One
can make no further progress unless some properties are satisfied.

\subsection{Strongly sufficient covariate}
Rigorous conditions must be investigated so as to identify ACE.
\begin{definition}
  $X$ is a covariate if:
\end{definition}
\begin{property}
  \label{pr:cov}
$$X \cip F_T.$$
\end{property}

That is, the distribution of $X$ is the same in any regime, be it
observational or interventional. In most cases, $X$ are attributes
determined prior to the treatment, for example, blood types and genes.

\begin{definition} 
  $X$ is a sufficient covariate for the effect of treatment $T$ on
  response $Y$ if, in addition to Property~\ref{pr:cov}, we have
\end{definition}
\begin{property}
  \label{pr:sufcov}
$$Y \cip F_T |(X,T).$$
\end{property}

Property \ref{pr:sufcov} requires that the distribution of $Y$, given
$X$ and $T$, is the same in all regimes. It can also be described as
``strongly ignorable treatment assignment, given $X$" \cite{pr:83}. We
assume that readers are familiar with the concept and properties of
directed acyclic graphs (DAGs). Then Properties \ref{pr:cov} and
\ref{pr:sufcov} can be represented by means of a DAG as
Fig.\ref{fig:suffcov}. The dashed arrow from $X$ to $T$ indicates that
$T$ is partially dependent on $X$, i.e., the distribution of $T$
depends on $X$ in the observational regime, but not in the
interventional regime where $F_T = t$.

\begin{figure}[!htbp]
  \sidecaption \leavevmode
  \includegraphics[bb = 70 150 310 730, scale=0.3,
  angle=270]{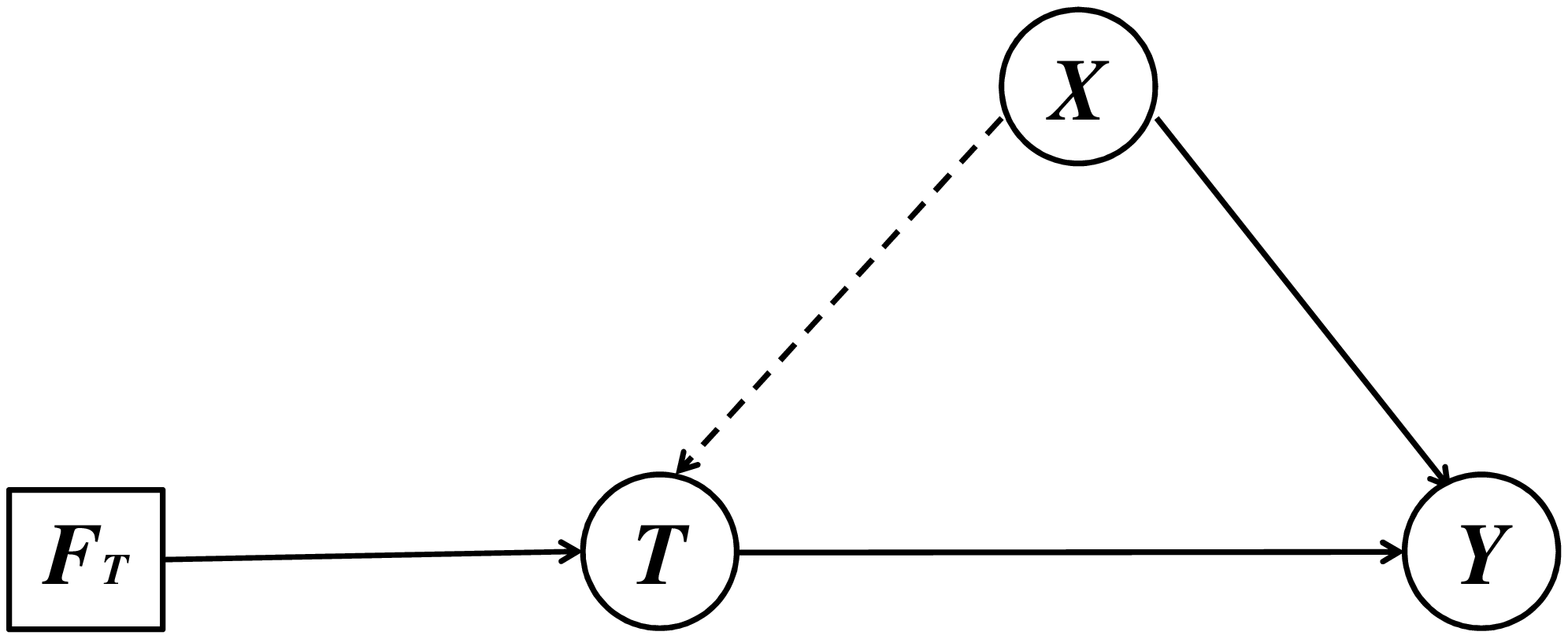}
  \caption{Sufficient covariate}
  \label{fig:suffcov}
\end{figure}

\begin{definition}
  \label{def:strong}
  $X$ is a {\em strongly sufficient covariate\/} if, in addition to
  Properties \ref{pr:cov} and \ref{pr:sufcov}, we have
\end{definition}
\begin{property}
  \label{pr:pos2}
  $\Pr_\emptyset(T =t \mid X) > 0$ with probabilility 1, for $t=0, 1$.
\end{property}

Property \ref{pr:pos2} requires that, for any $X=x$, both treatment
and control groups are observed in the observational regime.

\begin{lemma}
  \label{lem:pos1}
  Suppose $X$ is a strongly sufficient covariate.  Then, considered as
  a joint distributions for $(Y,X,T)$, $\Pr_t$ is absolutely
  continuous with respect to $\Pr_\emptyset$ (denoted by $\Pr_t \ll
  \Pr_\emptyset$), for $t=0$ and $t=1$. That is, for every event $A$
  determined by $(X,T,Y)$,
  \begin{equation}
    \Pr_\emptyset(A) = 0  \quad  \Longrightarrow  \quad \Pr_t(A) = 0.
  \end{equation}
  Equivalently, if an event $A$ occurs with probability 1 under the
  measure $\Pr_\emptyset$, then it occurs with probability 1 under the
  measure $\Pr_t$ ($t=0, 1$).
\end{lemma}

\begin{proof}
  Property \ref{pr:sufcov}, expressed equivalently as ${(Y,X,T)} \cip
  {F_T} | {(X,T)}$, asserts that there exists a function $w(X,T)$ such
  that
 $$\Pr_f(A \mid X, T) = w(X,T)$$
 almost surely (a.s.) in each regime $f = 0, 1, \emptyset$. Let
 $\Pr_\emptyset(A) = 0$.  Then a.s. $[\Pr_\emptyset]$,
$$0 = \Pr_\emptyset(A \mid X) = w(X,1)\Pr_\emptyset(T=1 \mid X) + w(X,0)\Pr_\emptyset(T=0 \mid X).$$                                          

By Property \ref{pr:pos2}, for $t=0,1$,
\begin{equation}
  \label{eq:ac1}
  w(X,t) = 0
\end{equation}
a.s. $[\Pr_\emptyset]$. As $w(X,t)$ is a function of $X$, it follows
that (\ref{eq:ac1}) holds a.s. $[\Pr_t]$ by Property
\ref{pr:cov}. Consequently,
\begin{equation}
  \label{eq:wxtpt} 
  w(X,T) = 0 \quad \as \;  [\Pr_t],
\end{equation} 
since a.s. $[\Pr_t]$, $T=t$ and $w(X,T) = w(X,t)$ for any bounded
function $w$. Then by (\ref{eq:wxtpt}),
   $$\Pr_t(A) = \e_t\{\Pr_t(A \mid X, T)\} = \e_t\{w(X,T)\} = 0.$$

 \end{proof}

 \begin{lemma}
   \label{lem:intezx}
   For any integrable $Z \preceq$ \footnote{The $\preceq$ symbol is
     interpreted as ``a function of". } $(Y, X, T)$, and any versions
   of the conditional expectations,
   \begin{equation}
     \label{eq:intezx}
     \e_t(Z \mid X) = \e_t(Z \mid X, T)  \quad \mbox{a.s.} \; [\Pr_t].
   \end{equation}
 \end{lemma}

\begin{proof}
  Let $j(X, T)$ be an arbitrary but fixed version of $\e_t(Z \mid X,
  T)$. Then $j(X, T) = j(X, t)$ a.s. $ [\Pr_t]$, and $j(X, t)$ serves
  as a version of $\e_t(Z \mid X, T)$ under $[\Pr_t]$. So
   $$\e_t (Z \mid X) = \e_t \{j(X, T) \mid X\} = \e_t \{j(X, t) \mid X\} = j(X, t)\quad \as \; [\Pr_t].$$ Thus $j(X, t)$ is a version of $\e_t (Z \mid X)$ under $[\Pr_t]$ and (\ref{eq:intezx}) follows.
 \end{proof}
 
 Since $\e_t (Z \mid X)$ is a function of $X$, then by Property
 \ref{pr:cov}, $j(X, t)$ is a version of $\e_t (Z \mid X)$ in any
 regime. Let $g(X, T)$ be some arbitrary but fixed version of
 $\e_\emptyset(Z \mid X, T)$.
 
 \begin{theorem}
   \label{thm:pos2new}
   Suppose that $X$ is a strongly sufficient covariate.  Then for any
   integrable $Z \preceq (Y, X, T)$, and with notation as above,
   \begin{equation}
     \label{eq:jxtgxt}
     j(X, t) = g(X, t)
   \end{equation}
   almost surely in any regime.
 \end{theorem}
 \begin{proof}
   By Property \ref{pr:sufcov}, there exists a function $h(X,T)$ which
   is a common version of $\e_f (Z \mid X, T)$ under $[\Pr_f]$ for
   $f=0, 1, \emptyset$. Then $h(X,T)$ serves as a version of
   $\e_\emptyset (Z \mid X, T)$ under $[\Pr_\emptyset]$, and a version
   of $\e_t (Z \mid X, T)$ under $[\Pr_t]$. As $j(X, T)$ is a version
   of $\e_t (Z \mid X, T)$,
  $$  j(X, T) = h(X, T) \quad \as \; [\Pr_t], $$
  and consequently
 $$  j(X, t) = h(X, t) \quad \as \; [\Pr_t].$$
 Since $j(X, t)$ and $h(X, t)$ are functions of $X$, by Property
 \ref{pr:cov}
 \begin{equation}
   \label{eq:jxthxtf}
   j(X, t) = h(X, t) \quad \as \; [\Pr_f]
 \end{equation}
 for $f=0,1,\emptyset$.  We also have that $g(X, T) = h(X, T) \;
 \as [\Pr_\emptyset]$, and so, by Lemma \ref{lem:pos1}, $\as
 [\Pr_t]$. Then $g(X, t) = h(X, t) \; \as [\Pr_t]$, where $g(X, t)$
 and $h(X, t)$ are both functions of $X$. By Property \ref{pr:cov},
 \begin{equation}
   \label{eq:gxthxtf}
   g(X, t) = h(X, t) \quad \as \; [\Pr_f]
 \end{equation}
 for $f=0,1,\emptyset$.  Thus (\ref{eq:jxtgxt}) holds by
 (\ref{eq:jxthxtf}) and (\ref{eq:gxthxtf}).
   
\end{proof}

\subsection{Specific causal effect}
\label{sec:backdoor}

Let $X$ be a covariate.
\begin{definition}
  The {\em specific causal effect \/} of $T$ on $Y$, relative to $X$,
  is
      $$\sce := \e_1(Y \mid X) - \e_0(Y \mid X).$$
\end{definition}
We annotate $\sce_X$ to express $\sce$ as a function of $X$ and write
$\sce(x)$ to indicate that $X$ takes specific value $x$.  Because it is
defined in the interventional regimes, $\sce$ has a direct causal
interpretation, \ie, $\sce(x)$ is the average causal effect in the
subpopulation with $X=x$.

Although we do not assume the existence of potential responses, when
this assumption is made we might proceed as follows.  Take $X$ to be
the pair ${\bf Y} = (Y(1), Y(0))$ of potential responses---which
is assumed to satisfy Property~\ref{pr:cov}.  Then $\e_t(Y \mid X) =
Y(t)$, and consequently
$$\sce_{\bf Y} = Y(1) - Y(0),$$
which is the definition of ``individual causal effect'', $\ice$, in
Rubin's causal model. Thus, although the formalisations of causality
are different, $\sce$ in Dawid's decision theoretic framework can be
reagarded as a generalisation of $\ice$ in Rubin's causal model.

We can easily prove that, for any covariate $X$, $\ace = \e(\sce_X)$,
where the expectation may be taken in any regime. Since by Property
\ref{pr:cov},
    $$\e_\emptyset\{\e_t(Y \mid X)\} = \e_t\{\e_t(Y \mid X)\} = \e_t(Y),$$ 
    for $t=0,1$. Thus by subtraction, $\ace = \e_f(\sce_X)$ for any
    regime $f = 0,1,\emptyset$ and therefore the subscript $f$ can be
    dropped.  Hence, $\ace$ is identifiable from observational data so
    long as $\sce_X$ is identifiable from observational data. If $X$
    is a strongly sufficient covariate, by Theorem \ref{thm:pos2new},
    $\e_t(Y \mid X)$ is identifiable from the observational regime. It
    follows that $\sce$ can be estimated from data purely collected in
    the observational regime.  Then $\ace$ expressed as
    \begin{equation}
      \label{eq:backdoor}
      \ace = \e_\emptyset(\sce_X)
    \end{equation}
    is identifiable, from the observational joint distribution of $(X,
    T, Y)$.  Formula (\ref{eq:backdoor}) is Pearl's ``back-door
    formula'' \cite{jp:00} because by the property of modularity,
    $\Pr(X)$ is the same with or without intervention on $T$ and thus
    can be taken as the distribution of $X$ in the observational
    regime.

    \subsection{Dimension reduction of strongly sufficient covariate}
    \label{sec:dimreduct}

\label{sec:reduction}
Suppose $X$ is a multi-dimensional strongly sufficient covariate. The
adjustment process might be simplified if we could replace $X$ by some
reduced variable $V \preceq X$, with fewer dimensions---so long as $V$
is itself a strongly sufficient covariate.  Now since $V$ is a
function of $X$, Properties~\ref{pr:cov} and \ref{pr:pos2} will
automatically hold for $V$.  We thus only need to ensure that $V$
satisfies Property~\ref{pr:sufcov}: that is,
\begin{equation}
  \label{eq:prop2v}
  Y \cip {F_T}  | {(V,T)}.
\end{equation}

Since two arrows initiate from $X$ in Fig.\ref{fig:suffcov}, possible
reductions may be naturally considered, on the pathways from $X$ to
$T$, and from $X$ to $Y$. Indeed, the following theorem gives two
alternative sufficient conditions for (\ref{eq:prop2v}) to
hold. However, (\ref{eq:prop2v}) can still hold without these
conditions.
\begin{theorem}
  \label{thm:either}
  Suppose $X$ is a strongly sufficient covariate and $V \preceq X$.
  Then $V$ is a strongly sufficient covariate if either of the
  following conditions is satisfied:
  \renewcommand{\theenumi}{{\rm{(\alph{enumi})}}}
  \begin{enumerate}
  \item \label{it:respsuff} {\bf Response-sufficient reduction:}

    \begin{equation}
      \label{eq:respsuff1}
      Y \cip X | {(V, F_T = t)},
    \end{equation}
    or
    \begin{equation}
      \label{eq:respsuff2}
      Y \cip X | {(V, T, F_T = \emptyset)},
    \end{equation}   
    for $t=0,1$. It is indicated in (\ref{eq:respsuff1}) that, in each
    interventional regime, $X$ contributes nothing towards predicting
    $Y$ once we know $V$. In other words, as long as $V$ is observed,
    $X$ need not be observed to make inference on $Y$. While
    (\ref{eq:respsuff2}) implies that in the observational regime,
    knowing $X$ is of no value of predicting $Y$ if $V$ and $T$ are
    known.
  \item \label{it:treatsuff} {\bf Treatment-sufficient reduction:}
    \begin{equation}
      \label{eq:treatsuff}
      T \cip X | {(V, F_T=\emptyset)}.
    \end{equation}
    That is, in the observational regime, treatment does not depend on
    $X$ conditioning on the information of $V$.
  \end{enumerate}
\end{theorem}

Proofs of the above reductions were provided in \cite{hg:10}. An
alternative proof of \ref{it:treatsuff} can be implemented graphically
\cite{hg:10}, which results in a DAG as Fig. \ref{fig:treatsuff}
\footnote{The hollow arrow head, pointing from $X$ to $V$, is used to
  emphasise that $V$ is a function of $X$.} off which
(\ref{eq:treatsuff}) and (\ref{eq:prop2v}) can be directly read.

\begin{figure}[!htbp]
  \sidecaption \leavevmode
  \includegraphics[bb = 70 150 310 730, scale=0.35,
  angle=270]{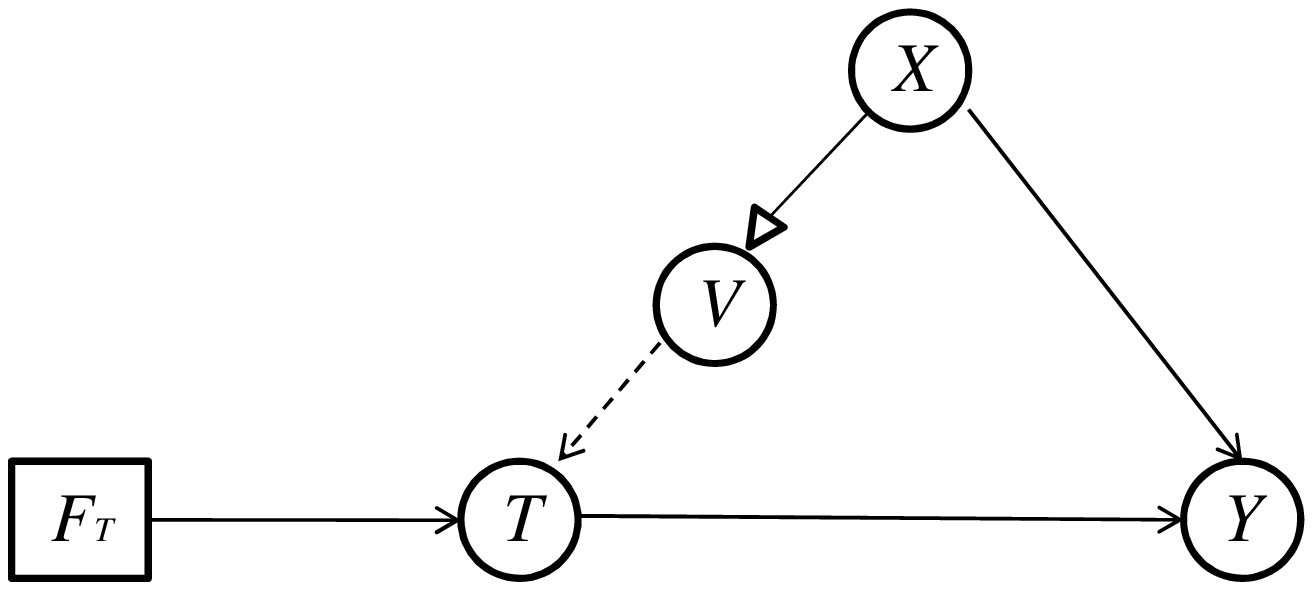}
  \caption{Treatment sufficient reduction}
  \label{fig:treatsuff}
\end{figure}

A graphical approach to \ref{it:respsuff} does not work since
Property~\ref{pr:pos2} is required.  However, while not serving as a
proof, Fig. \ref{fig:respsuff} conveniently embodies the conditional
independencies Properties \ref{pr:cov}, \ref{pr:sufcov} and the
trivial property $V \cip T | {(X, F_T)}$, as well as
(\ref{eq:prop2v}).

\begin{figure}[!htbp]
  \sidecaption \leavevmode
  \includegraphics[bb = 70 150 310 730, scale=0.35,
  angle=270]{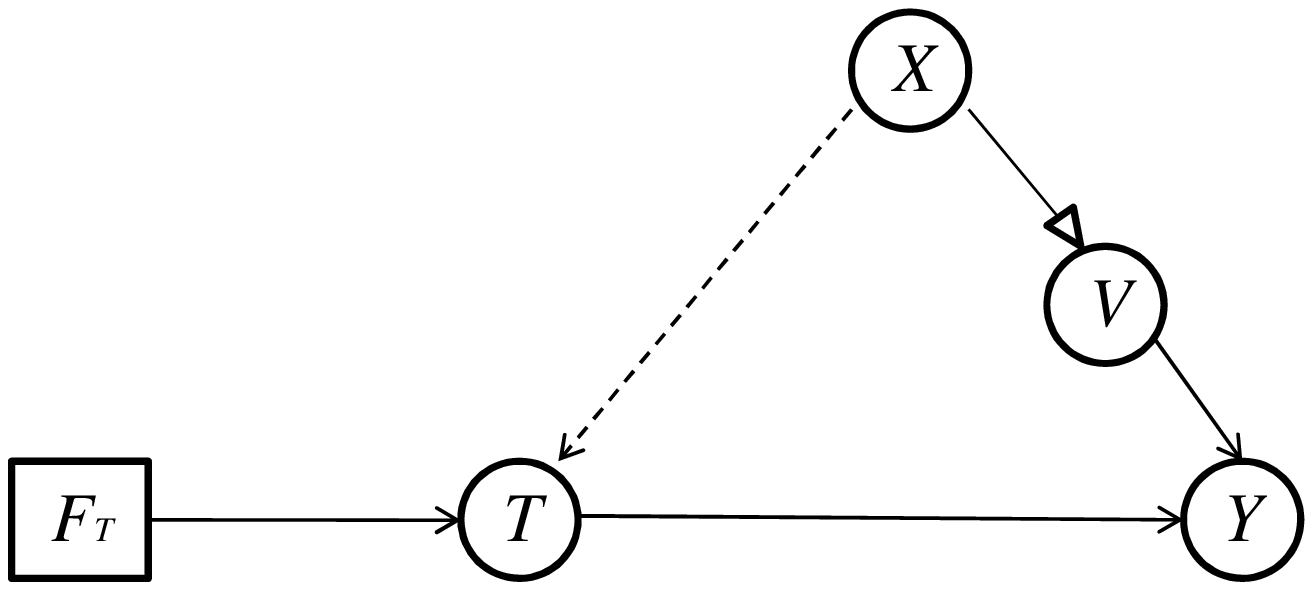}
  \caption{Response sufficient reduction}
  \label{fig:respsuff}
\end{figure}

\section{Propensity analysis}
\label{sec:ps}

Here we further discuss the treatment-sufficient reduction, which does
not involve the response. This brings in the concept of
\emph{propensity variable}: a minimal treatment-sufficient covariate,
for which we investigate the unbiasedness and precision of the
estimator of $\ace$. Also the asymptotic precision of the estimated
$\ace$, as well as the variation of the estimate from the actual data,
will be analysed. In a simple normal linear model that applied for covariate adjustment, two cases are
considered: homoscedasticity and heteroscedasticity. A non-parametric
approach -- subclassification will also be conducted, for different
covariance matrices of $X$ of the two treatment arms. The estimated
$\ace$ obtained by adjusting for multivariate $X$ and by adjusting for a scalar
propensity variable, will then be compared theoretically and through
simulations \cite{hg:10}.

\subsection{Propensity score and propensity variable}
\label{subsec:propvar}

The propensity score (PS), first introduced by Rosenbaum and Rubin, is
a balancing score \cite{pr:83}. Regarded as a useful tool to reduce
bias and increase precision, it is a very popular approach to causal
effect estimation. PS matching (or subclassification) method,
widely used in various research fields, 
exploits the property of
\emph{conditional (within-stratum) exchangeability}, whereby individuals with the same value of PS (or belonging to a group with similar values of
PS) are taken as comparable or exchangeable.
We will, however, mainly focus on the application of PS within a linear
regression. The definitions of the balancing score and PS given
below are borrowed from \cite{pr:83}.
\begin{definition}
  \label{def:bs}
  A \emph{balancing score} $b(X)$ is a function of $X$ such that, in
  the observational regime \footnote{Rosenbaum and Rubin do not define
    the balancing score and the PS explicitly for observational
    studies, although they do aim to apply the PS approach in such
    studies.}, the conditional distribution of $X$ given $b(X)$ is the
  same for both treatment groups. That is,
  $$X \cip T | {(b(X),F_T=\emptyset)}.$$
\end{definition}
It has been shown that adjusting for a balancing score rather than $X$
results in unbiased estimate of \ace, with the assumption of strongly
ignorable treatment assignment \cite{pr:83}. One can trivially choose
$b(X) = X$, but it is more constructive to find a balancing score to
be a many to one function.

\begin{definition}
  \label{def:ps}
  The \emph{propensity score}, denoted by $\Pi$, is the probability of
  being assigned to the treatment group given $X$ in the
  observational regime:
$$\Pi:= \Pr_\emptyset(T = 1\;|\;X).$$
\end{definition}

We shall use the symbol $\pi$ to denote a particular realisation of $\Pi$.
By (\ref{eq:treatsuff}) and Definitions \ref{def:bs} and \ref{def:ps},
we assert that PS is the coarsest balancing score.  For a subject
$i$, PS is assumed to be positive, \ie, $0 < \pi_{i} < 1$. Those
with the same value of PS are equally likely to be allocated to the
treatment group (or equivalently, to the control group), which
provides observational studies with the randomised-experiment-like
property based on measured $X$. This is because the characteristics of
the two groups with the same or similar PS are ``balanced". Therefore,
the scalar PS serves as a proxy of multi-dimensional variable $X$, and
thus, it is sufficient to adjust for the former instead of the latter.
In observational studies, PS is generally unknown because we do
not know exactly which components of X have impact on $T$ and how the
treatment is associated with them. However, we can estimate PS
from the observational data.

PS analysis for causal inference is based on a sequence of two stages:

\textbf{Stage 1:} PS Estimation. It is estimated by the observed
$T$ and $X$, and normally by a logistic regression of $T$ on $X$ for
binary treatment. Note that the response $Y$ is
irrelevant at this stage. Because we can estimate PS without
observing $Y$, there is no harm in finding an "optimal" regression 
model of $T$ on $X$ by repeated trials.

\textbf{Stage 2:} Adjusting for PS. Various adjustment approaches
have been developed, \eg, linear regression. If we are unclear about the
conditional distribution of $Y$ given $T$ and PS, non-parametric
adjustment such as matching or subclassification could be applied
instead.

Although two alternatives for dimension reductions have been provided, in practice, this type of
reduction may be more convenient in many cases. For example, 
certain values of the response may occur rarely and only after long observation periods
after treatment. In addition, it may sometimes be tricky to determine a "correct" form for a regression model of $Y$ on $X, T$ and $F_T$. Swapping the positions of $X$ and $T$, Equation (\ref{eq:treatsuff})
can be re-expressed as
\begin{equation}
  \label{eq:balance}
  X \cip T | {(V,F_T=\emptyset)},
\end{equation}
which states that the observational distribution of $X$ given $V$ is
the same for both treatment arms. That is to say, $V$ is a {\em
  balancing score\/} for $X$.

The treatment-sufficient condition~\ref{it:treatsuff} can be equivalently interpreted as
follows.  Consider the family ${\cal Q} = \{Q_0, Q_1\}$ consisting of
observational distributions of $X$ for the two groups $T=0$ and $T=1$.
Then Equation (\ref{eq:treatsuff}), re-expressed as (\ref{eq:balance}), says
that $V$ is a {\em sufficient statistic\/} (in the usual Fisherian
sense \cite{raf:25}) for this family.  In particular, a \emph{minimal}
treatment-sufficient reduction is obtained as a minimal sufficient
statistic for ${\cal Q}$: \ie, any variable almost surely equal to a
one-one function of the likelihood ratio statistic $\Lambda :=
q_1(X)/q_0(X)$, where $q_i(\cdot)$ is a version of the density of
$Q_i$.
  
\begin{definition}
  A \emph{propensity variable} is a minimal treatment-sufficient
  covariate, or a one-one function of the likelihood ratio statistic
  $\Lambda$.
\end{definition}

The concept of a propensity variable is derived from PS which is
related to $\Lambda$ in the following way:
\begin{equation}
  \label{eq:bayes}
  \Pi =\Pr_\emptyset(T=1 \mid X) = \theta\,\Lambda / ({1-\theta + \theta\Lambda}),
\end{equation}
where $0 < \theta := \Pr_\emptyset(T=1) < 1$ by Property
\ref{pr:pos2}.

It is entirely possible, from the above discussion, that a different
propensity variable will be obtained if we start from a different
strongly sufficient covariate.

\subsection{Normal linear model (homoscedasticity)}
The above theory will be illustrated by a simple example under linear-normal homoscedastic parametric assumptions.
\label{sec:nlm}
\subsubsection{Model construction}
\label{sec:modconstruct}
Suppose we have a scalar response variable $Y$, and a $(p \times 1)$
 strongly sufficient covariate $X$ that satisfies Properties
~\ref{pr:cov}, \ref{pr:sufcov} and \ref{pr:pos2}.  Let the conditional
distribution of $Y$ given $(X,T,F_T)$ be specified as:
\begin{equation}
  \label{eq:model2} 
  Y \mid (X, T, F_T)\sim \normal(d + \delta T + b'X, \phi),
\end{equation}
where the symbol $\sim$ stands for``is distributed as" and the symbol $\cal N$ stands for normal
distribution, with parameters $d$ and $\delta$ (scalar), $b$ ($p \times 1$), and
$\phi$ (scalar). Note that here and in the following models,
we assume no interactions between variables in $X$ although interactions can be formally dealt with via dummy variables.  Suppose $X$ is a
strongly sufficient covariate, then the coefficient $\delta$ of $T$ in
(\ref{eq:model2}) is the average causal effect $\ace$, which can be
easily proved as follows.
\begin{eqnarray*}
  \ace &=& \e(\sce_X) = \e\{\e_1(Y \mid X)\} - \e\{\e_0(Y \mid X)\}\\
  &=& \e(d + \delta + b'X) - \e(d + b'X) = \delta \qquad \qquad \by \quad (\ref{eq:model2}).
\end{eqnarray*}
It is readily seen that the specific causal effect $\sce_X$ is a
constant and equals $\delta$.

From (\ref{eq:model2}), the {\em linear predictor\/} $\lp := b\transp
X$ satisfies the conditional independence properties in
Condition~\ref{it:respsuff} of Theorem \ref{thm:either}. Thus, $\lp$
is a response-sufficient reduction of $X$, and $\e(Y \mid \lp, T) = d
+ \delta T + \lp$, with coefficient $\delta$ of $T$ that does not depend on the regime by virtue of the sufficiency condition.

Now assume that our model for the observational distribution of $(T,X)$ is as follows:
\begin{equation}
  \label{eq:modelt}
  \Pr_\emptyset(T = 1) = \theta
\end{equation}
\begin{equation}
  \label{eq:model1} 
  X  \mid (T, F_T=\emptyset) \sim \normal(\mu_T, \Sigma)
\end{equation} 
with parameters $\theta \in (0,1)$, $\mu_0$($p \times 1$), $\mu_1$ ($p \times
1$), and covariance matrix $\Sigma$ ($p \times p$, positive definite, identical in the two treatment groups). The corresponding marginal distribution of $X$ is a multivariate normal
mixture
\begin{equation}
  \label{eq:modelx} 
  X  \mid  F_T=\emptyset \sim (1-\theta)\, \normal(\mu_0, \Sigma) +
  \theta\, \normal(\mu_1, \Sigma),
\end{equation}
in the observational regime, and because we have assumed Property
\ref{pr:cov} to hold, also in the interventional regime.  
The observational distribution of $T$ given $X$ is given by
(\ref{eq:bayes}), with
\begin{eqnarray}
  \log \Lambda & = & \log\{\Pr_\emptyset(X \mid T=1)\} - \log\{\Pr_\emptyset(X \mid T=0)\}  \nonumber \\
  \label{eq:loglr}
  & = & - \frac{1}{2}(\mu_1\transp \Sigma^{-1}\mu_1 - \mu_0\transp \Sigma^{-1}\mu_0) + \ld ,
\end{eqnarray}
where
\begin{equation}
  \label{eq:v}
  \ld : = \gamma\transp X,
\end{equation}
with
\begin{equation}
  \label{eq:gamma}
  \gamma :=  \Sigma^{-1} (\mu_1 -\mu_0).
\end{equation}
$\ld$ is Fisher's {\em linear discriminant\/} \cite{kvm:79}, best
separating the pair of multivariate normal observational distributions
for $X \mid T=0$ and $X \mid T=1$.

Suppose $V$ is a linear sufficient covariate -- a linear function of
$X$ that is itself a sufficient covariate.  We have proved that the
coefficient of $T$ in the observational linear regression of $Y$ on
$T$ and $V$ is $\delta$ \cite{hg:10}.  From (\ref{eq:loglr}) we see
that $\ld$ is a propensity variable which is a linear strongly
sufficient covariate. We deduce that under the given distributions,
the coefficient of $T$ in the observational regression of $Y$ on $T$
and $\ld$ is $\delta$.

\begin{theorem}
  \label{thm:gendis}
  The coefficient of $T$ in the linear regression of $Y$ on $(T,\ld)$
  is the same as that in the linear regression of $Y$ on $(T,X)$.
\end{theorem}
Theorem \ref{thm:gendis} states that it is algebraically true that $X$
and Fisher's linear discriminant $\ld$ generate identical coefficient
of $T$ in linear regressions, which does not have a direct link to the
regimes and causality whatsoever. In our linear normal model, $\delta$
is interpreted as $\ace$ and can be identified from the observational
data simply because we have assumed that $X$ is a strongly sufficient
covariate.  Applying Theorem \ref{thm:gendis} to the empirical
distribution of $(Y, T, {X})$ from a sample, we deduce \corref{sample}
as follows.
\begin{corollary}
  \label{cor:sample}
  Suppose we have data on $(Y, T, X)$ for a sample of individuals.
  Let $\ld^*$ be the sample linear discriminant for $T$ based on $X$.
  Then the coefficient of $T$ in the sample linear regression of $Y$
  on $T$ and $\ld^*$ is the same as that in the sample linear
  regression of $Y$ on $T$ and $X$.
\end{corollary}
Rosenbaum and Rubin \cite{pr:83} (\S~3.4) also give this result with a
brief non-causal argument: whenever the sample dispersion matrix is
used in both the form of $\ld$ and regression adjustment, the
estimated coefficient of $T$ must be the same.

As discussed \cite{hg:10}, here is a paradox: we regard
adjustment for the propensity variable as an adjustment for the treatment
assignment process, by regressing $Y$ on $T$ and the estimated
propensity variable $\ld^*$. However, from the result of
\corref{sample}, it appears that what we actually adjust for is the full set of covariates $X$,
which makes the treatment assignment process completely irrelevant.

\subsubsection{Precision in propensity analysis}
\label{subsec:precision}
One might intuitively think that the precision of the estimated $\ace$
would be improved if we were to adjust for a scalar variable --- the
sample-based propensity variable $\ld^*$, rather than $p$-dimensional
variable $X$.  However, \corref{sample} tells us that adjusting for
$\ld^*$ does not increase the precision of our estimator. In fact,
whether one adjusts for $\ld^*$ and for all the $p$ predictors makes {\em
  absolutely no difference\/} to our estimate, and thus, to its
precision.  Similar conclusions have been drawn in \cite{jh:98, wcw:04,
  ss:07}. Our intuition is that the increased precision obtained by regressing on
$V$ is offset by the overfitting error involved in selecting $V$.

Previous evidence \cite{jr:92, kh:03, dr:92} supports the claim that the
estimated propensity variable outperforms the true propensity
variable. That is, adjusting for the former yields higher precision of
the estimated $\ace$ than the latter. These two types of adjustment correspond to regressing
 $Y$ on $(T, \ld)$ and on $(T, \ld^*)$ in our model and both provide
an unbiased estimator of $\ace$. The claim obviously cannot be always
valid by simply considering a special case: $\ld = \lp$, because by
\corref{sample}, regressing on $\ld^*$ is the same as adjusting for
$\lp^*$, which by the Gauss-Markov theorem will be less precise than
regressing on the true linear predictor $\lp$ (or equivalently $\ld$).
Nevertheless, the claim is likely to hold when $\ld$ is not highly
correlated with $\lp$ because $\ld$ is a less precise response
predictor.

\subsubsection{Asymptotic variance analysis}
\label{sec:asyvar}

To gain a closer insight into the variance of the estimated $\ace$, by
adjusting for the true propensity variable $\pv$ (if known) and the
estimated propensity variable $\epv$, we consider a toy example in
which the parameters in (\ref{eq:model2}), (\ref{eq:modelt}) and
(\ref{eq:model1}) are set as follows: $$p = 2, \quad P_\idle(T=1) =
\theta \in (0, 1), \quad b = (b_1, b_2)',$$ the covariance matrix
$\Sigma$ is diagonal with identical entries $\tau$, and
\begin{equation}
  \label{eq:ex2t}
  \e(X_2 \mid T=1) = \e(X_2 \mid T=0) = \e(X_2)
\end{equation}
By the setting of $\Sigma$, we see that
\begin{equation}
  \label{eq:x1x2t}
  X_1 \cip X_2\;|\;T. 
\end{equation}

It is also clear that the true $\pv$ is just $X_1$, by minimal
treatment-sufficient reduction and related equations (\ref{eq:loglr}),
(\ref{eq:v}), (\ref{eq:gamma}). The conditions according to our model
setting are expressed by a DAG as Fig. \ref{fig:x1x2}.

\begin{figure}[!htbp]
  \begin{center}
    \leavevmode
    \includegraphics[bb = 50 230 240 560, scale=0.4,
    angle=270]{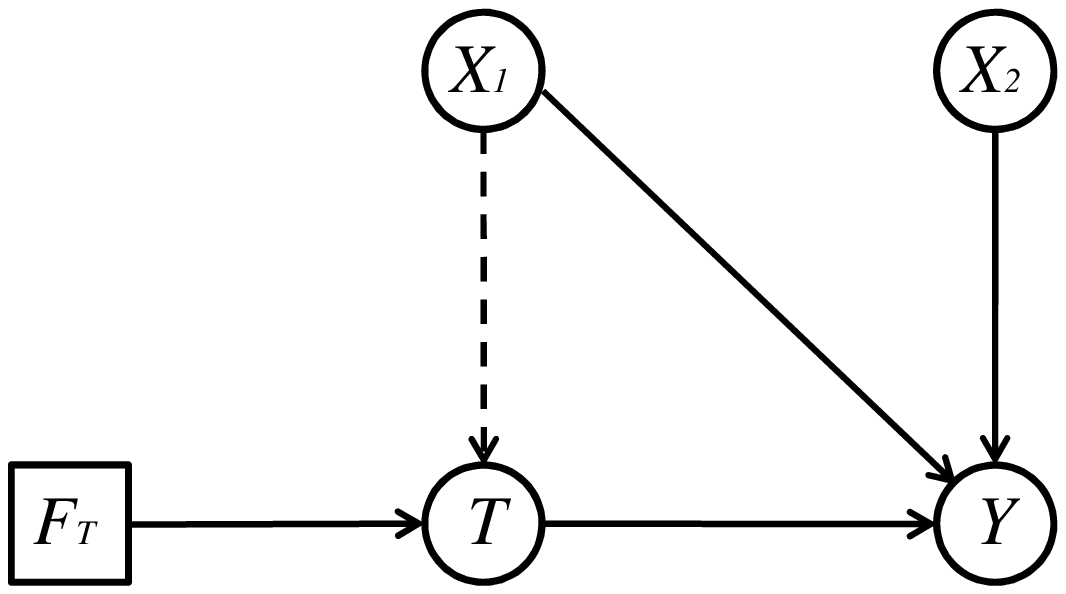}
    \caption{Propensity variable $X_1$ and response predictor $X = (X_1, X_2)$}
    \label{fig:x1x2}
  \end{center}
\end{figure}

In practice, all the parameters are unknown, and consequently the
exact form of $\pv$ is not known. What one would normally do is
adjust for the whole set of the observed $X$, which is equivalent to
adjusting for $\ld^*$ (or $\epv$) in the linear regression
approach by \corref{sample}. In particular, two linear regressions are
considered as follows:

\vspace{2mm}

\underline{$M_{0}$:} $Y$ on ($T, X$),

\vspace{2mm}

\underline{$M_{1}$:} $Y$ on ($T, X_1$).

\vspace{2mm}

Then the design matrix is $(1, T, X_1, X_2)'$ for $M_{0}$ and $(1, T,
X_1)'$ for $M_{1}$. Let $\widehat{\beta_{M_{0}}}$ and
$\widehat{\beta_{M_{1}}}$, respectively, be the least square
estimators of the parameters in $M_{0}$ and $M_{1}$. The asymptotic
variance of $\widehat{\beta_{M_{0}}}$ for sample size $n$ is then
given as:
$$\asyvar(\widehat{\beta_{M_{0}}}) = \frac{A^{-1} \var(Y \mid T, X)}{n} = \frac{A^{-1} \phi}{n},$$
where

\[ A = \left( \begin{array}{cccc}
    1 & \theta & \e(X_1) & \e(X_2) \\
    \theta & \theta & \e(TX_1) & \e(TX_2) \\
    \e(X_1) & \e(TX_1) & \e({X_1}^2) & \e(X_1X_2) \\
    \e(X_2) & \e(TX_2) & \e(X_1X_2) & \e({X_2}^2) \end{array}
\right).\]

By solving $A^{-1}$ and extract the (2, 2)th element which is variance
multiplier of the coefficient of $T$, we have that
$$\asyvar(\widehat{\delta_{M_{0}}}) = \frac{(W_{X_1X_1}W_{X_2X_2} - {W_{X_1X_2}}^2) \phi}{n\theta(1 - \theta)(V_{X_1X_1}V_{X_2X_2} - {V_{X_1X_2}}^2)},$$
where
$$W_{X_1X_2} =  \e(X_1X_2) -  \e(X_1)\e(X_2) = \cov(X_1,X_2),$$
$$V_{X_1X_2} = \e(X_1X_2) - \theta \e(X_1 \mid T=1)\e(X_2 \mid T=1) - (1 - \theta) \e(X_1 \mid T=0)\e(X_2 \mid T=0),$$
$$W_{X_1X_1} = \e({X_1}^2) - [\e(X_1)]^2 = \var(X_1),$$
$$W_{X_2X_2} = \e({X_2}^2) - [\e(X_2)]^2 = \var(X_2),$$
and
$$V_{X_1X_1} = \e({X_1}^2) - \theta[\e(X_1 \mid T=1)]^2 - (1 - \theta)[\e(X_1 \mid T=0)]^2,$$
$$V_{X_2X_2} = \e({X_2}^2) - \theta[\e(X_2 \mid T=1)]^2 - (1 - \theta)[\e(X_2 \mid T=0)]^2.$$
By (\ref{eq:ex2t}),
$$V_{X_2X_2} = \var(X_2) = W_{X_2X_2}$$
and
$$V_{X_1X_2} = \cov(X_1,X_2) = W_{X_1X_2},$$
where, by (\ref{eq:x1x2t}),
$$\cov(X_1,X_2) = \e\{\cov(X_1 \mid T, X_2 \mid T)\} + \cov\{\e(X_1 \mid T), \e(X_2 \mid T)\} = 0.$$
Hence,
\begin{equation}
  \label{eq:asyvarm01}
  \asyvar(\widehat{\delta_{M_{0}}}) = \frac{\phi \var(X_1) / [n\theta(1 - \theta)]}{\e({X_1}^2) - \theta[\e(X_1 \mid T=1)]^2 - (1 - \theta)[\e(X_1 \mid T=0)]^2}.
\end{equation}
For $M_{1}$, by (\ref{eq:x1x2t}),
\begin{eqnarray}
  \asyvar(\widehat{\delta_{M_{1}}}) &=& \frac{W_{X_1X_1}}{n\theta(1 - \theta)V_{X_1X_1}} \cdot \var(Y \mid T, X_1) \nonumber \\
  &=& \frac{W_{X_1X_1}}{n\theta(1 - \theta)V_{X_1X_1}} \cdot \{\phi + {b_2}^2\var(X_2 \mid T, X_1)\} \nonumber \\
  \label{eq:asyvarm1} 
  &=&  \frac{(\phi + {b_2}^2\tau) \var(X_1) / [n\theta(1 - \theta)]}{\e({X_1}^2) - \theta[\e(X_1 \mid T=1)]^2 - (1 - \theta)[\e(X_1 \mid T=0)]^2}.
\end{eqnarray}

Comparing (\ref{eq:asyvarm01}) and (\ref{eq:asyvarm1}), we have that
$\asyvar(\widehat{\delta_{M_{0}}}) <
\asyvar(\widehat{\delta_{M_{1}}})$ unless $X_2$ is random noise rather
than the linear predictor \ie, $b_2 = 0$ which equalises the two
asymptotic variances.

\begin{lemma}
  \label{lem:asymp}
  Under the given distributional assumptions (\ref{eq:model2}),
  (\ref{eq:modelt}) and (\ref{eq:model1}), suppose the propensity
  variable $\ld$ is not the same as the linear predictor $\lp$, and
  $\ld$ is independent of variables that are merely response
  predictors. Then the asymptotic variance of the estimated $\ace$ from the
  linear regression by adjusting for the estimated propensity variable
  $\ld^*$ is more precise than that by adjusting for the population
  propensity variable $\ld$.
\end{lemma}

\subsubsection{Simulations}
\label{sec:simhomo}
Simulations are carried out for numerical illustration.  Suppose we
have the following true values for the paramters in (\ref{eq:model2}),
(\ref{eq:modelt}) and (\ref{eq:model1}): $p = 2, d = 0, \delta = 0.5,
b=(0,1)\transp, \phi = 1, \theta = 0.5, \mu_1 = (1, 0)\transp, \mu_0 =
(0, 0)\transp$, $\Sigma=I_2$.

Then the population linear predictor is $\lp=X_2$, with
$$ Y \mid (X, T, F_T)\sim \normal(\frac{1}{2} T + X_2, 1),$$
while the population linear discriminant $\ld=X_1$ which is not
predictive to $Y$.  Since for any regime $f = 0, 1, \emptyset$,
$$\e_f(Y \mid X_1, T) = \e_f\{\e_f(Y \mid X, T) \mid X_1, T\} = \frac{1}{2} T$$
and
$$\var_f(Y \mid X_1, T) = \e_f\{\var_f(Y \mid X, T) \mid X_1, T\} + \var\{\e_f(Y \mid X, T)\mid X_1, T\} = 2.$$
The conditional distribution of $Y$ given $(X_1, T)$, for any regime,
is then given by
  $$Y \mid (X_1, T, F_T)\sim \normal(\frac{1}{2} T, 2).$$

  To investigate the performance of the population-based as well as
  sample-based LP and PV, we now
  consider four linear regression models:

  \quad $M_0$: $Y$ on $T$ and $X$ ($X = (X_1, X_2)$),
    
  \quad $M_1$: $Y$ on $T$ and $X_1$,

  \quad $M_2$: $Y$ on $T$ and $X_2$,

  \quad $M_3$: $Y$ on $T$ and $\ld^*$,

  where $M_0$ is the full model with all parameters unknown. In $M_1$,
  by setting $b_2 = 0$, the true linear discriminant $\ld = X_1$ is
  fitted. While fitting the true linear predictor $\lp = X_2$,
  equivalent to setting $b_1 = 0$, we get $M_2$. Note that all these
  models are ``true". For $M_1$ the true value of $b_1$ is 0, and the
  true residual variance is 2, as against 1 for $M_0$ and
  $M_2$. Finally, for any dataset with no information of parameters,
  we construct the estimated propensity variable $\ld^*$, and then fit
  the model $M_3$.

  In each model $M_k$, for $k=0,1,2,3$, the least-squares estimator
  $\hat{\delta_k}$ is unbiased for $\delta = 0.5$. By the Gauss-Markov
  theorem and \corref{sample},
$$\var(\hat{\delta_0}) = \var(\hat{\delta_3}) \geq \var(\hat{\delta_2}).$$
Asymptotically, we have that $\var{.\emph{asy}}(\hat{\delta_0}) =
\var{.\emph{asy}}(\hat{\delta_3}) = 5/n,$
$\var{.\emph{asy}}(\hat{\delta_2})= 4/n,$ and
$\var{.\emph{asy}}(\hat{\delta_1})= 10/n$.  It is
indeed asymptotically less precise to adjust for PV than for its estimate in our model, which is in accordance with
\lemref{asymp}.

For the sample analysis, 200 simulated datasets are generated,
each of size $n = 20$. Shown in Fig. \ref{fig:lrpvpdecm} are the
empirical distributions of $\hat{\delta_k}$ for all four
models. Unsurprisingly, in terms of precision (from high to low),
first comes the $\lp$; next is the estimated propensity variable
$\ld^*$ (or the estimated linear predictor $\lp^*$), or equivalently, $X$ (=
($X_1, X_2$)); and last comes the true propensity variable $\ld = X_1$.

\begin{figure}[htp]
  \begin{center}
    \includegraphics[scale=0.45, angle=270]{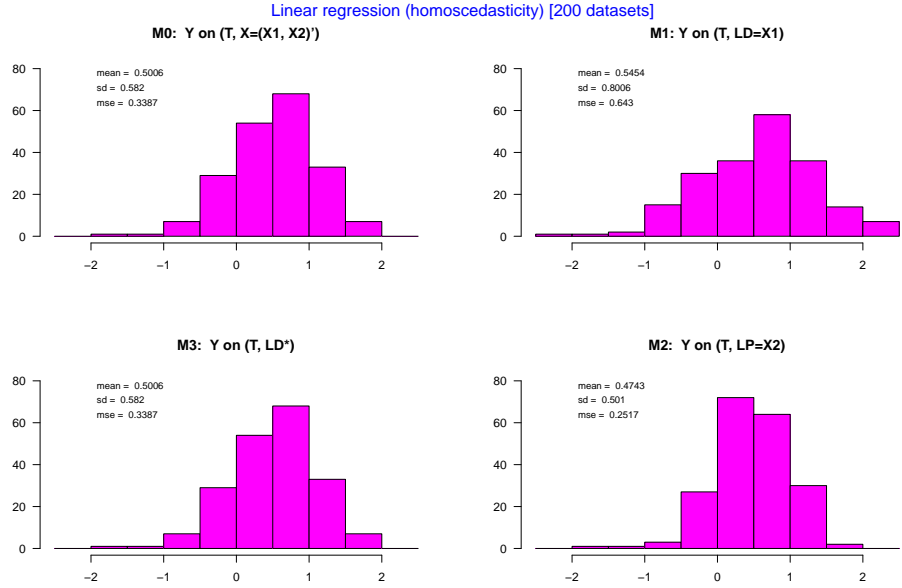}
    \caption{Estimates of $\ace$ by regression on (clockwise):
      1. $X_{1}$ and $X_{2}$. \quad 2. population linear discriminant
      (propensity variable) $X_{1}$. \quad 3.  population linear
      predictor $X_{2}$. \quad 4. estimated linear discriminant
      (propensity variable) $\ld^*$ .}
    \label{fig:lrpvpdecm}
  \end{center}
\end{figure}

\subsection{Normal linear model (heteroscedasticity)}
\label{sec:diffcov}
Investigation in the homoscedasticity case is simple because PV is
equivalent to LD, where linearity makes analysis straightforward. If
covariance matrices of the conditional distribution of $X$ for the two
treatment groups are not identical, it turns out that adjusting for
$\pv$ is not appropriate.

Suppose now that, keeping all other distributional assumptions of
\secref{nlm} unchanged, (\ref{eq:model1}) is re-specified as:
  $$X  \mid (T, F_T=\emptyset) \sim \normal(\mu_T, \Sigma_T)$$
  with different covariance matrices $\Sigma_0$ and $\Sigma_1$ for
  $T=0$ and $T=1$.  The distribution of $X$ in all regimes then
  becomes
  $$X \mid F_T \sim (1-\theta)\, \normal(\mu_0, \Sigma_0) + \theta\, \normal(\mu_1, \Sigma_1).$$
  Accordingly,
$$  \log \Lambda =  c + \qd$$
where
$$c = - \frac{1}{2}\{\log(\det\Sigma_1) - \log(\det\Sigma_0) + \mu_1\transp \Sigma_1^{-1}\mu_1 - \mu_0\transp \Sigma_0^{-1}\mu_0\}$$
and
\begin{equation}
  \label{eq:qd}
  \qd := \left(\Sigma_1^{-1}\mu_1 -\Sigma_0^{-1}\mu_0\right)\transp X - \half
  X\transp \left(\Sigma_1^{-1}-\Sigma_0^{-1}\right)X.
\end{equation}
$\qd$ is the {\em quadratic discriminant\/} including a linear term
and a quadratic term of $X$, distinguishing the observational
distributions of $X$ given $T=0, 1$.  We see that $\qd$ is a minimal
treatment-sufficient covariate, and thus a PV but no longer a linear
function of $X$.

Because of the balancing property of PS (or PV), it now follows
that $\ace = \e(\sce_\qd),$ with
$$\sce_\qd = \e_1( Y \mid \qd) - \e_0( Y \mid \qd).$$
Since $\qd$ is quadratic in $X$, $Y$ is no longer linear in $\qd$, the
coefficient of $T$ by adjusting for $\pv$ (= $\qd$) in the linear
regression does not provide exact $\ace$. However, as computation of
the expectations in the above formula is non-trivial, one might wish
to replace $\e_\emptyset( Y \mid T, \qd)$ by the linear regression of
$Y$ on $(T, \qd)$, and approximate the estimated \ace.  Alternatively,
one can take non-parametric approaches such as matching or
subclassification on $\qd$ \cite{pr:84}.  A number of papers on
various matching approaches for causal effects have been collected in
\cite{dr:06}. More recently, statistical software becomes available
for multivariate and PS matching in R \cite{js:11}.

Now we discuss subclassifications and linear regressions based on
$\qd$, compared to linear regressions based on $\lp$ and $\ld$. The linear
discriminant is again in the form
$$\ld =(\mu_1-\mu_0)\transp \Sigma^{-1} X,$$
but with $\Sigma = (1-\theta) \Sigma_0 + \theta \Sigma_1,$ the sum of the
weighted dispersion matrices of the two treatment groups.  Form the
formulae of $\qd$ and \ld, we conclude that it is LD that
comprises all variables with expectations depending on 
$T$. In a DAG representation of this scenario, each of such variables must have an arrow pointing to
$T$. However, the genuine PV (= $\qd$) may depend on all the
components of $X$, according to its quadratic term in
(\ref{eq:qd}). Only with homoscedasiticity, PV is equivalent to
$\ld$ and includes all variables associated with $T$.

Although $\ld$ is not a sufficient covariate here, \thmref{gendis}
still applies. It enables us to identify $\ace$ from the linear
regression of $Y$ on $(T,\ld)$, which is equivalent to the linear
regression of $Y$ on $(T,X)$.  However, other authors claim that only if
$\ld$ is highly correlated with PS, adjustment for $\ld$
works well in regressions \cite{pr:83}. This may attribute to
different scenarios considered, \ie, in our model $Y$ is linearly
related to $X$ while non-linear in $X$ in theirs.

\subsubsection{Simulations}
\label{sec:simulations2}

Simulated data is based on the above model, with the parameters:
$p=20$, $d= 0$, $\delta = 0.5$, $\theta = 0.5$, $b = (0, 1, \ldots,
0)\transp,$ $\mu_0 = (0, \ldots, 0)\transp,$ and $\mu_1 = (0.5, 0, 0,
\ldots, 0)\transp.$ Also, $\Sigma_0$ is set, diagonally, to $0.8$ for
the first ten entries and to $1.3$ for the remaining entries, and
$\Sigma_1$ the identity matrix.

We then have, for the population, that
$$\ld = \frac{5}{9}X_{1},$$
$$\pv = \qd = \frac{1}{2}X_{1} + \frac{1}{8} \sum_{i=1}^{10} X_{i}^2 - \frac{3}{26} \sum_{j=11}^{20} X_{j}^2,$$
and $\lp = X_{2}$.  By estimating
$\mu_0$ and $\mu_1$, $\Sigma_0$ and $\Sigma_1$ from observed data, we
can compute sample-based $\ld^*$ and $\qd^*$.

\begin{figure}[!htbp]
  \begin{center}
    \leavevmode
    \includegraphics[scale=0.45, angle=270]{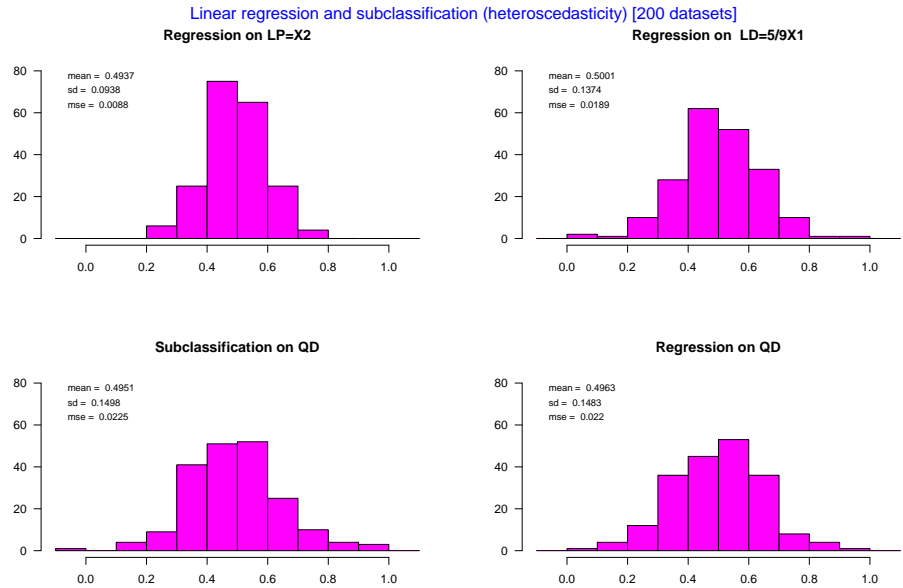}
    \caption{Estimates of $\ace$ by 4 different methods (clockwise):
      1.  Regression on population linear predictor $\lp =
      X_{2}$. \quad 2. Regression on population linear discriminant
      $\ld = \frac{5}{9}X_{1}$. \quad 3. Regression on population
      quadratic discriminant (propensity variable) $\qd$. \quad
      4. Subclassification on $\qd$.}
    \label{fig:pslrsubtruenew}
  \end{center}
\end{figure}

The results from 200 simulated datasets, each of size 500, are given
in Fig. \ref{fig:pslrsubtruenew}.  The first three plots (clockwise)
are from the linear regressions of $Y$ on, respectively $(T,X_{2})$,
$(T, \ld)$, and $(T, \qd)$.  The last plot is the result of
subclassification on PV ( = $\qd$). That is, 500
observations are divided into 5 subclasses with equal number of
observations in each, based on the values of $\qd$. Within each
subclass, units from the two treatment groups are roughly comparable
such that the average difference of the response may be interpreted as
the estimated $\sce$.  Then $\ace$ is estimated by summing over
 $\sce$s, each weighted by $1/5$.  Note that the sample size has increased, since we must have at least one observation
for each treatment in each subclass.

\begin{figure}[!htbp]
  \begin{center}
    \leavevmode
    \includegraphics[scale=0.45, angle=270]{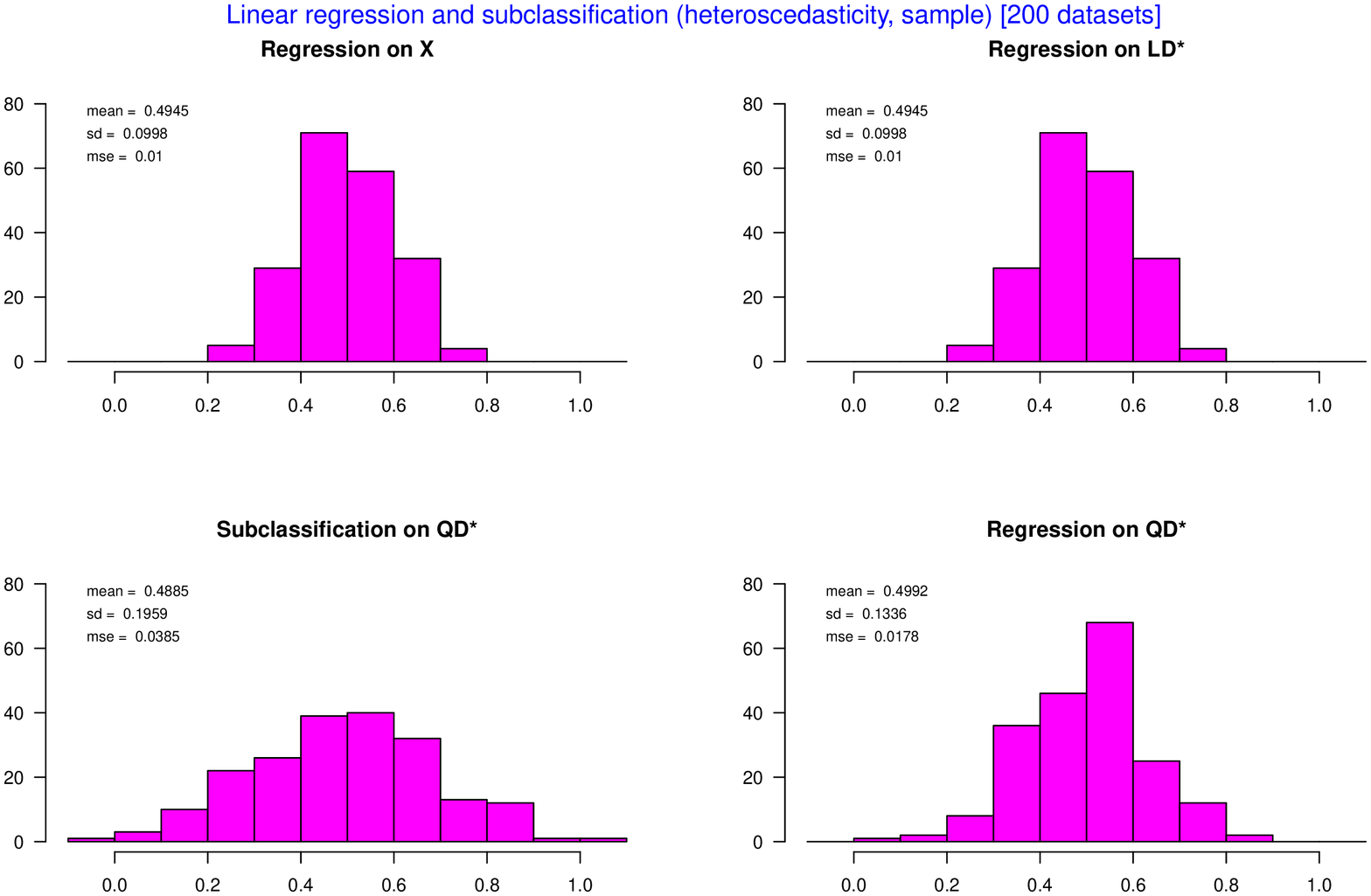}
    \caption{Estimates of $\ace$ by 4 different methods (clockwise):
      1. Regression on sufficient covariate $X$.  \quad 2. Regression
      on sample linear discriminant $\ld^*$. \quad 3. Regression on
      sample quadratic discriminant (propensity variable)
      $\qd^*$. \quad 4. Subclassification on $\qd^*$.}
    \label{fig:pslrsubnew}
  \end{center}
\end{figure}

Since $\ld$ and $\qd$ are practically unknown, they need be estimated
from the observed data. Also, we do not know exactly the response
predictors or the confounders, full set of the observed $X$ may have to be
used for analysis.

Fig. \ref{fig:pslrsubnew} gives the results from the same 200 datasets
as above. Again, the first three plots are the results of linear
regressions of $Y$, but on, respectively, $(T, X)$, $(T, \ld^*)$, and
$(T, \qd^*)$, where $\ld^*$ and $\qd^*$ are the sample linear and
quadratic discriminants.  Shown in the last plot is the result of
subclassification on EPV ( = $\qd^*$). Unsurprisingly, by
comparing the mean, standard deviation and mean squared error
of the estimated $\ace$, regression of $Y$ on ($T,\lp = X_{2}$) comes
the best among all eight approaches in Fig. \ref{fig:pslrsubtruenew}
and Fig. \ref{fig:pslrsubnew}. Regressing on ($T$,
$X$) is no better than regressing on ($T$, $X_{2}$) because all
variables except $X_{2}$ in $X$ are not predictors but noise of
$Y$. In confirmation of the theory in \secref{modconstruct},
regressing on $\ld^*$, rather than on $X$, has absolutely no
effect on the estimated $\ace$.  $\ld^*$ outperforms $\ld$ because the
latter does not contain the response predictor.  Regressions on
\ld, $\qd$, and on $\qd^*$ are roughly equal, because apart from
$X_{1}$, the distributions of the remaining 19 variables are
identical, with rather small multipliers. Thus, the two quadratic
terms in $\qd$ are roughly the same, and $\qd \approx
\frac{1}{2}X_{1}$ works approximately as a function of a single
variable $X_{1}$. Last comes subclassification on the quadratic PV,
particularly when it is estimated.

\subsection{Propensity analysis in logistic regression}
\label{sec:logit}

As already investigated, propensity analysis in linear regression is
fairly straightforward. In many cases, however, response $Y$ is not
linear in $X$. We know that despite its name, generalised linear model
(GLM) is not a linear model, because it is a non-linear function of
the response that is linearly related to its predictors.  Logistic
regression is widely applied as a type of GLM if the response is
binary. For example, doctors often record the outcome of a surgery on
a patient as either ``cured" or ``not cured". Next, a logistic model is used in our illustrative study.

\subsubsection{Model construction}
\label{sec:logitmodel}
For simplicity, suppose that $Y,T (1 \times 1)$ and $X (p \times 1)$
are all binary and components of $X$ are mutually independent, The
joint distribution of $(F_T, X, T, Y)$ is constructed as follows:

\begin{equation}
  \label{eq:binx}
  X  \mid F_T \sim Ber(\pi)
\end{equation}
\begin{equation}
  \label{eq:bintx}
  \logit\{\Pr_\emptyset(T \mid X)\} = c + a \transp X
\end{equation}
\begin{equation}
  \label{eq:binytx}
  \logit\{\Pr_f(Y \mid T, X)\} = d + \delta T + b \transp X,
\end{equation}
for $f = 0,1,\emptyset$; and $\pi$ is ($p \times 1$). Property
\ref{pr:pos2} and $P_f(Y = 1 \mid T, X) \in (0, 1)$ are required such
that (\ref{eq:bintx}) and (\ref{eq:binytx}) are well-defined.

It is immediately seen that $X$ is a strongly sufficient covariate and
\begin{eqnarray*}
  \ace &=& \e_\emptyset\{\e_1(Y\mid X)\} - \e_\emptyset\{\e_0(Y\mid X)\}\\
  &=& \e_\emptyset\{\Pr_\emptyset(Y\mid  T=1, X)\} - \e_\emptyset\{\Pr_\emptyset(Y\mid T=0, X)\}\\
  &=& \e_\emptyset\{\frac{1}{1 + e^{-(d + \delta + b \transp X)}} - \frac{1}{1 + e^{-(d + b \transp X)}}\}.
\end{eqnarray*}
If the parameters are set as follows:
\begin{equation}
  \label{eq:par1}
  p = 3, \quad \pi=(\pi_1, \pi_2, \pi_3)\transp, \quad a = (a_1, a_2, 0)\transp,\quad b=(0, b_2, b_3)\transp, 
\end{equation} 
then the response predictor is $b_2X_2 + b_3X_3$ and $\pv = a_1X_1 + a_2X_2.$ The conditional independence properties in our
model can be read off Fig. \ref{fig:logit}.
\begin{figure}[h]
  \sidecaption \leavevmode
  \includegraphics[scale=0.2, angle=270]{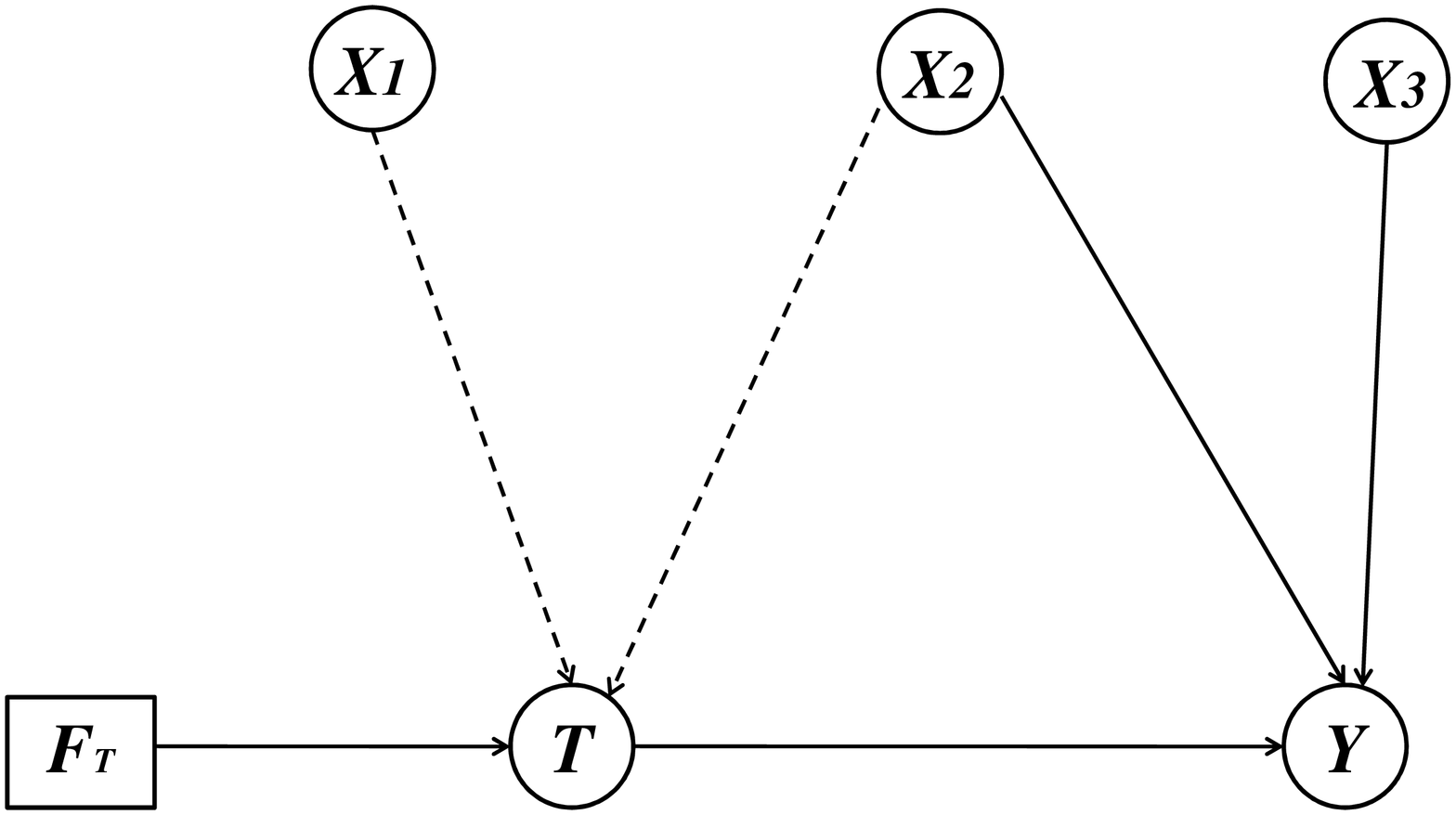}
  \caption{DAG for the logistic model}
  \label{fig:logit}
\end{figure}
Then we have that
 $$  \logit\{P(Y \mid T, X_2, X_3)\}  =  \logit\{P(Y \mid T, X)\}  = d + \delta T + b_2X_2 + b_3X_3,$$
 and
 \begin{eqnarray*}
   P(Y = 1 \mid T, X_1, X_2) &=& P(Y = 1 \mid T, X_2)\\
   &=& \e\{\frac{1}{1+e^{-(d + \delta T + b_2X_2 + b_3X_3)}} \mid T, X_2\}\\
   \label{eq:ytx2}
   &=&  \frac{\pi_3}{1+e^{-(d + \delta T + b_2X_2 + b_3)}} + \frac{1-\pi_3}{1+e^{-(d + \delta T + b_2X_2)}}
 \end{eqnarray*}
 which does not depend on $X_1$.  And we have that
 \begin{eqnarray}
   \ace &=& \pi_2\pi_3 \{\frac{1}{1 + e^{-(d + \delta + b_2 + b_3)}} - \frac{1}{1 + e^{-(d + b_2 + b_3)}}\}   \nonumber \\ 
   &\;\;\;\;\;+& (1-\pi_2)\pi_3 \{\frac{1}{1 + e^{-(d + \delta + b_3)}} - \frac{1}{1 + e^{-(d + b_3)}}\} \nonumber \\
   &\;\;\;\;\;+& \pi_2(1-\pi_3) \{\frac{1}{1 + e^{-(d +\delta + b_2)}} - \frac{1}{1 + e^{- (d + b_2)}}\}   \nonumber \\
   \label{eq:logitace}
   &\;\;\;\;\;+& (1-\pi_2)(1-\pi_3) \{\frac{1}{1 + e^{-(d + \delta)}} - \frac{1}{1 + e^{-d}}\},
 \end{eqnarray}
 which is determined by $d, \delta, b_2, b_3$, $\pi_2$ and $\pi_3$. This
 extremely simple example, with only three components of $X$ that are all binary, already results in a complicated
 form for \ace, which would be even worse for high dimensional $X$ and various types of variables. Next, instead of simulation, we conduct propensity
 analysis on real data.

 \subsubsection{Propensity analysis of custodial sanctions study}
 We illustrate the method with the aid of a study involving 511
 subjects sentenced to prison in 1980 by the California Superior
 Court, and 511 offenders sentenced to probation following conviction
 for certain felonies \cite{gb:13}. These probationers were matched to the
 prisoners on county of conviction, condition offence type and risk of
 imprisonment quantitative index, so as to bring into the final sample
 the most serious offenders on probation and the least serious
 offenders sentenced to prison. The structure of this study
 corresponds to the (partially matched) case-control design. In fact, 
 this is analogous to the regression discontinuity designs where only
 observations near the cut-off of the risk score are included for causal effect analysis \cite{gwi:07}. 
 We were to compare the average causal effect of judicial sanction (probation
 or prison) on the probability of re-offence. We specify variables as
 follows.

 \begin{itemize}
 \item Treatment $T$: taking values 0 (probation) and 1 (prison);
 \item Response $Y$: occurrence of recidivism (re-offence);
 \item Pre-treatment variable $X$: including 17 carefully selected non-collinear
   variables that we can reasonably assume to make $X$ a strongly
   sufficient covariate.
 \end{itemize}

 Simple random multiple imputation by bootstrapping (R package: mi)
 was applied to deal with missing data. We then considered two
 logistic regressions for the imputed data:
 \begin{enumerate}
 \item $Y$ on $(T, X)$, where $X$ includes all the 17 variables.

   \vspace{2mm}

 \item $Y$ on $(T,$ EPS), where EPS is the propensity score estimated
   from the logistic regression of $T$ on all the 17 variables. In selecting these variables, we took advantage of the possibility of
   trying various sets of covariates in the model, without inflating the type I error since these regressions do not involve the response information.
    The distribution densities of the two treatment groups are shown in
   Fig. \ref{fig:psdens}, where we see a large overlapping area.
 \end{enumerate}

\begin{figure}[h]
  \sidecaption \leavevmode
  \includegraphics[scale=0.4, angle=270]{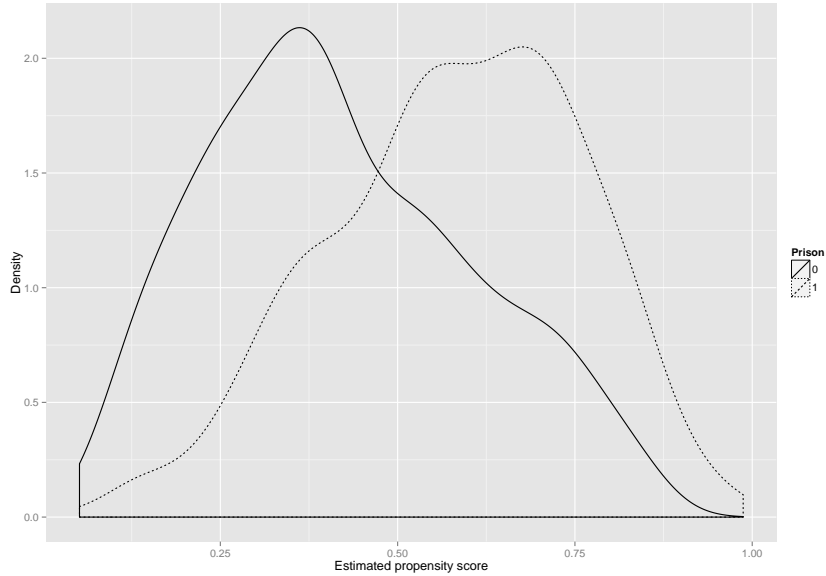}
  \caption{Distribution density comparison of the estimated propensity
    score: prison vs. probation. }
  \label{fig:psdens}
\end{figure}

Shown in Tab. \ref{tab:logit.result} are the results. In this case, regression on the full set of
$X$ and on the estimated PS makes little difference, since the summary
statistics from the two approaches are quite similar. Although the
negative values of both the coefficients imply reduced re-offence for
the imprisonment, they are not statistically significant.

\begin{table}[h]
  \caption{Coefficients of judicial sanction (``prison" with respect to ``probation") from logistic regressions: 1. $Y$ on ($T, X$); 2. $Y$ on ($T$, EPS)}
  \vspace{1 mm}
  \centering
  \begin{tabular}{l rrr}
    \hline

    {Regression} & Coefficient & Standard Error & p-value \\[0.5ex]
    \hline

    {$Y$ on ($T, X$)}      & -0.1631 & 0.1579 & 0.3014 \\
    {$Y$ on ($T$, EPS)} & -0.1713 & 0.1503 & 0.2545 \\[0.5ex]
    \hline
  \end{tabular}
  \label{tab:logit.result}
\end{table}

\section{Double robustness}

Since the underlying response regression model (RRM): $Y \mid (X, T,
F_T = \idle)$ and the propensity model (PM): $T \mid (X, F_T = \idle)$
are most likely unknown, one may specify parametric models based on
previous experience.  Moreover, as discussed in
\S~\ref{sec:dimreduct}, a strongly sufficient covariate can be reduced
by two alternative approaches from specified models, which enables
estimating $\ace$ by either method as follows:
\begin{enumerate}
\item Adjustment for response predictors from correctly specified RRM;
  
  \vspace{2mm}
  
\item Adjustment for a $\pv$ (or $\ps$) from correctly specified PM,
  either in response regression (if RRM is correctly specified), or
  otherwise, by non-parametric approaches, \eg, matching.
\end{enumerate}
Due to lack of knowledge, it may well be that \emph{at least one} model is
misspecified. Little could be done if both models are wrong. Thus, our
interest is to find a single estimator that produces a good estimate,
given that at least one model is correct.

$\ace$ is normally estimated from the observed data. Suppose there are
$n$ individuals in an observational study. Observations $(x_i, t_i,
y_i)$, where $i = 1, .., n$, are generated from the joint distribution
$(X_i, T_i, Y_i)$ that are independent and identically
distributed. The estimation of the $\ace$ requires estimates of the
expected response for both treatment groups assigned by
intervention. We have already demonstrated that, within the $\dt$
framework, $\ace$ is identifiable from pure observational data if $X$
is a strongly sufficient covariate. Here, $X$ is again assumed to be
strongly sufficient and thus satisfies Properties \ref{pr:cov},
\ref{pr:sufcov} and \ref{pr:pos2}.

\subsection{Augmented inverse probability weighted estimator}

To construct the augmented inverse probability weighted (AIPW)
estimator, we discuss two scenarios:
\begin{itemize}
\item \textbf{Correct RRM}: Suppose that we know the RRM.  For
  convenience, we write $\e_t(Y)$ as $\mu_t$, so
  \begin{equation}
    \label{eq:ety}
    \mu_t = \e_\emptyset[\e_\emptyset(Y\;|\;X,T = t)]
  \end{equation}
  since $X$ is strongly sufficient. Hence, in observational studies,
  $\e_\emptyset(Y\;|\;X,T = t)$ is an unbiased estimator of $\mu_t$,
  for $t=0, 1$. Consequently, $\e_\emptyset(Y\;|\;X,T = 1) -
  \e_\emptyset(Y\;|\;X,T = 0)$ is an unbiased estimator of $\ace$.

  \vspace{2mm}

\item \textbf{Correct PM}: Consider that the PM is correct, \ie,
  $\pi(X) = \Pr_\emptyset(T = 1 \mid X)$.
\end{itemize}
\begin{lemma}
  \label{lem:aceht}
  Suppose that the propensity model is correct and that $X$ is a strongly sufficient
  covariate. Then
\end{lemma}
\begin{equation}
  \label{eq:ace2}
  \ace = \e_\emptyset\{\frac{T}{\pi(X)}Y\} - \e_\emptyset\{\frac{1-T}{1-\pi(X)}Y\},
\end{equation} 
where $\e_\emptyset\{\frac{T}{\pi(X)}Y\} = \mu_1$ and
$\e_\emptyset\{\frac{1-T}{1-\pi(X)}Y\} = \mu_0.$

\begin{proof}
  \begin{eqnarray*}
    \e_\emptyset\{\frac{T}{\pi(X)}Y\} &=& \e_\emptyset\{\e_\emptyset(\frac{T}{\pi(X)}Y\;|\;X)\} = \e_\emptyset\{\frac{1}{\pi(X)}\e_\emptyset(TY\;|\;X)\} \\
    &=& \e_\emptyset\{\frac{1}{\pi(X)}\e_\emptyset(Y\;|\;X, T=1)\Pr_\emptyset(T=1\;|\;X)\} \\
    &=& \e_\emptyset\{\e_\emptyset(Y\;|\;X, T=1)\} = \mu_1   \qquad \qquad \qquad \by \; (\ref{eq:ety}).
  \end{eqnarray*}
\end{proof}
It automatically follows that $\e_\emptyset\{\frac{1-T}{1-\pi(X)}Y\} =
\mu_0.$ By \lemref{aceht}, we see that, under the observational
regime, $\frac{T}{\pi(X)}Y$ and $\frac{1-T}{1-\pi(X)}Y$ are unbiased
estimators of $\mu_1$ and $\mu_0$ respectively.

One may have noticed that the two terms for $\ace$ in (\ref{eq:ace2})
are similar with the Horvitz-Thompson (HT) estimator for sample
surveys \cite{ht:52}. They are, however, different in various
aspects. The aim of HT estimator is to estimate the mean of a finite
population $Y_1, ..., Y_N$, denoted by $\mu = N^{-1}\sum_{i=1}^N Y_i$,
from a stratified sample of size $n$ drawn without replacement. For
$i=1,...,N$, let $\Delta_i$ be binary sampling indicator ($\Delta_i =
1$: unit $i$ is in sample; $0$: unit $i$ is not in sample), and
$\pi_{i}$ be the probability that unit $i$ being drawn in the
sample. Then HT estimator is given by:
\begin{equation}
  \label{eq:ht}
  \muhat_{HT} = N^{-1}\sum_{i=1}^N \frac{\Delta_{i}}{\pi_{i}}Y_{i},
\end{equation} 
where $\pi_i$ is pre-specified, and thus known in a sample survey
design. But the propensity model $\pi(X)$ in (\ref{eq:ace2}) is
normally unknown. Moreover, HT estimator is applied to estimate the
mean of a finite population, while $\acehat$ is used to estimate the
mean of a superpopulation \footnote{In causal system, finite number of
  individuals in a study is called ``population", which can be
  regardes as a sample from a larger "superpopulation" of
  interest}. HT estimator depends on pre-specified sampling scheme,
but observations involved in $\acehat$ are generated from, and thus
are dependent on, the joint distribution of $(X, T, Y)$ in the
observational regime. Nevertheless, both HT estimator and $\acehat$
are formed by means of the inverse probability weights $1/\pi_i$ or
$1/\pi(X)$. In fact, HT estimator is also termed as the inverse
probability weighted (IPW) estimator.

Sample surveys are closely related to missing data because the
information is missing for those not sampled. So IPW estimator is
frequently used in missing data models in the presence of partially
observed response \cite{hb:05, jc:06, jk:07}. As counterfactuals are
also regarded as missing data, IPW estimator can be used in the
potential response framework with half observed information, to make
causal inference of treatment effect under the assumptions of
``strongly ignorable treatment assignment": $(Y(0),
Y(1)){\cip}T\;|\;X$ and ``no unobserved confounders" \cite{hb:05,
  zt:07}.

\subsubsection{Augmented inverse probability weighted estimator}
From above discussion, there exists an unbiased estimator of $\ace$ if
either RRM or PM is correct. However, unknown RRM and PM makes it
impossible to decide whether they are correct. Nevertheless, the
augmented inverse probability weighted (AIPW) estimator can be
constructed by combining the two models in the following alternative
forms:
\begin{eqnarray}
  \label{eq:e1yaipw1}
  \muhat_{1, \aipw} &=&  m(X) +  \frac{T}{\pi(X)}(Y - m(X)) \nonumber \\
  \label{eq:e1yaipw2}
  &=& \frac{T}{\pi(X)}Y + [1- \frac{T}{\pi(X)}]m(X),
\end{eqnarray}
and similarly,
\begin{eqnarray}
  \label{eq:e0yaipw1}
  \muhat_{0, \aipw} &=&  m(X) +  \frac{1-T}{1-\pi(X)}(Y - m(X)) \nonumber \\
  \label{eq:e0yaipw2}
  &=& \frac{1-T}{1-\pi(X)}Y + [1- \frac{1-T}{1-\pi(X)}]m(X),
\end{eqnarray}

where $m(\cdot)$ and $\pi(\cdot)$ are arbitrary functions of $X$.  As
also indicated in its name, $\muhat_{t, \aipw}$ is the sum of the IPW
estimator and an augmented term.
\begin{lemma}
  \label{lem:dr}
  Suppose that $X$ is a strongly sufficient covariate. The estimator
  $\muhat_{t, \aipw}$ has the property of \textbf{double
    robustness}. That is, $\muhat_{t, \aipw}$ is an unbiased estimator
  of the population mean given $T=t$ by intervention, if either
  $\pi(X) = p_\idle(T = 1\;|\;X)$ or $m(X) = \e_\idle(Y\;|\;X, T=t)$.
\end{lemma}

\begin{proof}
  \vspace{2mm} By similarity, we only give proof of $\muhat_{1,
    \aipw}$. Consider the following two scenarios:

  \vspace{2mm}

  \textbf{Scenario 1: } $\pi(X) = p_\idle(T = 1\;|\;X)$ and $m(X)$ is
  an arbitrary function of $X$.

  It is easily seen that $\muhat_{1, \aipw}$ is unbiased, from the
  proof of \lemref{aceht}. Since conditional on $X$, the last term in
  (\ref{eq:e1yaipw2}) vanishes when we take expectation of $\muhat_{1,
    \aipw}$ in the observational regime.
 
  \vspace{3 mm}

  \textbf{Scenario 2: } $m(X)=\e_\idle(Y\;|\;X, T=1)$ and $\pi(X)$ is
  an arbitrary function of $X$.

  By (\ref{eq:e1yaipw1}), we have that
  \begin{eqnarray*}
    \e(\muhat_{1, \aipw}) &=& \e[m(X) + \frac{T}{\pi(X)}(Y - m(X))]\\
    &=& \e\{\e[m(X)\;|\;X]\} + \e\{\e[\frac{T}{\pi(X)}(Y - m(X))\;|\;X]\}\\
    &=&  \e[m(X)] + \e\{\frac{\e(TY \mid X) - m(X)\e(T \mid X)}{\pi(X)}\}\\
    &=&  \e[m(X)] = \mu_1 \qquad \qquad \qquad \qquad \qquad \by \; (\ref{eq:ety}).
  \end{eqnarray*}
\end{proof}
Indeed, if either $\pi(X) = p_\idle(T = 1\;|\;X)$ or $m(X) =
\e_\idle(Y\;|\;X, T=1)$, not necessarily both, $\muhat_{1, \aipw}$ is
unbiased. Consequently,
$$\acehat_{\aipw} = \muhat_{1, \aipw} - \muhat_{0, \aipw}.$$

\begin{thm}
  \label{thm:drace}
  Suppose that $X$ is a strongly sufficient covariate. Then the AIPW
  estimator $\acehat_{\aipw}$ is doubly robust.
\end{thm}

To prove \thmref{drace}, we simply apply the fact that both
$\muhat_{1, \aipw}$ and $\muhat_{0, \aipw}$ are doubly robust, so is
their difference.

\subsection{Parametric models}
Suppose that we specify two parametric working models: the propensity
working model $\pi(X; \alpha)$ and the response regression working
model $m(T, X; \beta)$. Then by (\ref{eq:e1yaipw2}) and
(\ref{eq:e0yaipw2}), we have, for the estimated $\e_1(Y)$ and
$\e_0(Y)$, that

\begin{equation}
  \label{eq:mu1}
  \muhat_{1,\aipw} = n^{-1}\{\sum_{i=1}^n \frac{T_{i}}{\pi(X_{i}; \hat{\alpha})}Y_{i} + [1-\frac{T_{i}}{\pi( X_{i}; \hat{\alpha})}]m(1, X_{i};\hat{\beta})\}
\end{equation}
and
\begin{equation}
  \label{eq:mu0}
  \muhat_{0,\aipw} = n^{-1}\{\sum_{i=1}^n \frac{1-T_{i}}{1-\pi(X_{i}; \hat{\alpha})}Y_{i} + [1-\frac{1-T_{i}}{1-\pi( X_{i}; \hat{\alpha})}]m(0, X_{i};\hat{\beta})\}
\end{equation}
respectively. Therefore, by (\ref{eq:mu1}) and (\ref{eq:mu0}), we have
that
\begin{eqnarray}
  \acehat_{AIPW} &=& \hat{\mu}_{1,AIPW} - \hat{\mu}_{0,AIPW} \nonumber \\
  \label{eq:aipw}
  &=& n^{-1}\{\sum_{i=1}^n [\frac{T_{i}}{\pi(X_{i}; \hat{\alpha})} - \frac{1-T_{i}}{1-\pi(X_{i}; \hat{\alpha})}](Y_{i}-m(T_{i}, X_{i};\hat{\beta})\},
\end{eqnarray}
which is doubly robust, \ie, $\acehat_{AIPW}$ is a consistent and
asymptotically normal estimator of $\ace$ if either of the working
models is correctly specified.

\subsubsection{Discussion}
Kang and Schafer \cite{jk:07} state that there are various ways to
construct an estimator which is doubly robust. In our view, they are
essentially the same, \ie, it must be in the same (or similar) form of
$\aipw$ estimator which is constructed by combining RRM and PM. Other
constructions proposed in \cite{jk:07} are just variations of $\aipw$
estimator. For example, in (\ref{eq:ht}), instead of using $N$ as
denominator for each unit, they use normalised weights $\sum_{i=1}^N
\frac{\Delta_{i}}{\pi_{i}}$. Such normalised weights are especially
useful for precision improvement in the case that subjects with very
small probabilities of being sampled are actually drawn from the
population. Because if $N$ is used as the weight, these subjects will
influence the estimated average response enormously, and consequently,
result in poor precision.

Kang and Schafer \cite{jk:07} have also investigated the precision
performance of an doubly robust estimator when both $\pi(X)$ and
$m(X)$ are moderately misspecified. They state that ``in at least some
settings, two wrong models are not better than one". This seems
obvious because the performance of this estimator will depend on the
degree of misspecification of both models. This can be easily analysed
in theory but far more complicated in practice, as one can not have a
good control of specifying models $\pi(X)$ and $m(X)$ based on limited
observed data and previous experience (if any). Therefore, it would be
difficult to measure to what extent the specified models are different
from the true ones.

\subsection{Precision of $\acehat_{AIPW}$}
\subsubsection{Known propensity score model}
\label{sec:eff}
We already see that $\acehat_{AIPW}$ is an unbiased and doubly robust
estimator of $\ace$. Then how can we choose an arbitrary function
$m(X_i)$ to minimise the variance of $\acehat_{AIPW}$ given correct
PM? Suppose that in an experiment, we know $\pi(X_i) = \Pr(T_i =
1\;|\;X_i)$. Then in terms of the variance, we have that
\begin{eqnarray*}
  \var(\acehat_{AIPW}) &=& \var\{n^{-1}[\sum_{i=1}^n (\frac{T_{i}}{\pi(X_{i})} - \frac{1-T_{i}}{1-\pi(X_{i})})(Y_{i}-m(X_{i})]\}\\
  &=& n^{-2}\{\var[\sum_{i=1}^n (\frac{T_{i}}{\pi(X_{i})} - \frac{1-T_{i}}{1-\pi(X_{i})})Y_{i}]\\
  &\;\;\;\;\; +& \var[\sum_{i=1}^n (\frac{T_{i}}{\pi(X_{i})} - \frac{1-T_{i}}{1-\pi(X_{i})})m(X_{i})]\\
  &\;\;\;\;\; -& 2\cov[\sum_{i=1}^n (\frac{T_{i}}{\pi(X_{i})} - \frac{1-T_{i}}{1-\pi(X_{i})})Y_{i}, \sum_{i=1}^n (\frac{T_{i}}{\pi(X_{i})} - \frac{1-T_{i}}{1-\pi(X_{i})})m(X_{i})]\}\\
  &=& n^{-2}\{\var(\acehat_{HT}) + \e[\sum_{i=1}^n \frac{m^2(X_{i})}{\pi(X_{i})(1-\pi(X_{i}))}]\\
  &\;\;\;\;\; -& 2\e[\sum_{i=1}^n \frac{m(X_{i})\mu_{1i}}{\pi(X_{i})(1-\pi(X_{i}))}-\frac{m(X_{i})(\mu_{1i}-\mu_{i})}{(1-\pi(X_{i}))^2}]\\
  &=& n^{-2}\{\var(\acehat_{HT}) + \e[\sum_{i=1}^n \frac{m^2(X_{i})}{\pi(X_{i})(1-\pi(X_{i}))}\\
  &\;\;\;\;\; -& 2\sum_{i=1}^n\{ \frac{\mu_{1i}}{\pi(X_{i})(1-\pi(X_{i}))}-\frac{\mu_{1i}-\mu_{i}}{(1-\pi(X_{i}))^2}\}m(X_{i})]\},
\end{eqnarray*}
where $\mu_{1i} = \e_\idle(Y_{i}\;|\;X_{i}, T_{i} = 1)$ and $\mu_{i} =
\e_\idle(Y_{i}\;|\;X_{i}).$

By minimising the quadratic function of $m(X_{i})$ in the expectation,
it follows that
\begin{eqnarray}
  \label{eq:mimin1}
  m(X_{i}) &=& [1-\pi(X_{i})]\mu_{1i}+\pi(X_{i})\mu_{0i} \nonumber \\
  \label{eq:mimin2}
  &=& [1-\pi(X_{i})]\e_\idle(Y_{i}\;|\;X_{i}, T_{i} = 1) + \pi(X_{i})\e_\idle(Y_{i}\;|\;X_{i}, T_{i} = 0),
\end{eqnarray}
which minimises the variance of $\acehat_{AIPW}$ among all functions
of $X_{i}$.  In fact, if either $\pi(X_{i}) = p_\idle(T_{i} =
1\;|\;X_{i})$, or (\ref{eq:mimin2}) holds, $\acehat_{AIPW}$ is
unbiased, and thus is doubly robust.

Let $m_{1}(X_{i})$ and $m_{0}(X_{i})$ denote the regressions of $Y$ on
$X_{i}$ for the two treatment groups in the observational regime. It
is unnecessary to require that $m_{1}(X_{i}) =
\e_\idle(Y_{i}\;|\;X_{i}, T_{i} = 1)$ and that $m_{0}(X_{i}) =
\e_\idle(Y_{i}\;|\;X_{i}, T_{i} = 0)$. As long as $m(X_{i})$ is
specified as the sum of the weighted expectations as in the form of
(\ref{eq:mimin2}), $m(X_{i})$ minimises the variance of the estimated
$\ace$.

Same result is obtained in \cite{dr:08} as (\ref{eq:mimin2}), by
minimising a weighted mean squared error of $m(X_{i})$. We now discuss
an alternative approach provided in \cite{dr:08}. Let $\tilde{Y_{i}}$
denote a weighted response in a form as follows:
\begin{equation}
  \tilde{Y_{i}} = [\{\frac{1}{\pi(X_{i})} -1\}T_{i} + \{\frac{1}{1-\pi(X_{i})} - 1\}(1-T_{i})]Y_{i}.
\end{equation} 
Then by (\ref{eq:mimin2}), it follows that
\begin{eqnarray*}
  \label{eq:alternative}
  m(X_{i}) &=& \frac{1-\pi(X_{i})}{\pi(X_{i})}\e_\idle(Y_{i}\;|\;X_{i}, T_{i} = 1)\Pr(T = 1\;|\;X)\\
  &\;\;\;\;\; +& \frac{\pi(X_{i})}{1-\pi(X_{i})}\e_\idle(Y_{i}\;|\;X_{i}, T_{i} = 0)\Pr(T = 0\;|\;X)\\
  &=& \frac{1-\pi(X_{i})}{\pi(X_{i})}\e_\idle(T_{i}Y_{i}\;|\;X_{i}) + \frac{\pi(X_{i})}{1-\pi(X_{i})}\e_\idle[(1-T_{i})Y_{i}\;|\;X_{i}]\\
  &=& \e_\idle\{[\frac{1-\pi(X_{i})}{\pi(X_{i})}T_{i} + \frac{\pi(X_{i})}{1-\pi(X_{i})}(1-T_{i})]Y_{i}\;|\;X_{i}\}\\
  &=& \e_\idle\{[(\frac{1}{\pi(X_{i})} -1)T_{i} + (\frac{1}{1-\pi(X_{i})} - 1)(1-T_{i})]Y_{i}\;|\;X_{i}\}\\
  \label{eq:ystar}
  &=& \e_\idle(\tilde{Y_{i}}\;|\;X_{i}),
\end{eqnarray*}
where $m(X_{i})$ is obtained by simply regressing $\tilde{Y_{i}}$ on
$X_{i}$, rather than regressing $Y_{i}$ on both $X_{i}$ and
$T_{i}$. However, an obvious disadvantage of this approach is its low
precision. When individuals with the PS close to 0 are actually in the
treatment group and/or those with the PS close to 1 are actually
assigned to the control group, the weights $1/\pi(X_{i})$ or
$1/(1-\pi(X_{i}))$ of these units will be very large, which leads to
corresponding responses being highly influential, which is dangerous.
In fact, it may be even worse than the HT estimator as we will see
next.

To show the difference of these approaches, we have implemented Monte
Carlo computations for four estimators of $\acehat_{AIPW}$:

\begin{enumerate}

\item by (\ref{eq:mimin2}) with $\e_\idle(Y_{i}\;|\;X_{i}, T_{i} = 1)$
  and $\e_\idle(Y_{i}\;|\;X_{i}, T_{i} = 0)$ estimated by regressing
  $Y_{i}$ on $(X_{i}, T_{i})$.

  \vspace{2 mm}

\item by (\ref{eq:mimin2}) with $\e_\idle(Y_{i}\;|\;X_{i}, T_{i} = 1)$
  and $\e_\idle(Y_{i}\;|\;X_{i}, T_{i} = 0)$ estimated by regressing
  $Y_{i}$ on $X_{i}$ for the treatment group and control group
  separately.

  \vspace{2 mm}

\item by Horvitz-Thompson approach, i.e. without covariate adjustment.

  \vspace{2 mm}

\item by regression of $\tilde{Y_{i}}$ on $X_{i}$ as in
  (\ref{eq:ystar}).

\end{enumerate}

The results of simulated 100 datasets are shown in
Fig. \ref{fig:dreff}. The first two approaches give similar
results. That is, we can estimate $\e_\idle(Y_{i}\;|\;X_{i}, T_{i} =
1)$ and $\e_\idle(Y_{i}\;|\;X_{i}, T_{i} = 0)$ either simultaneously
from the response regression on the treatment and $X$, or separately
from the response regression only on $X$ for each treatment group. As
expected, the last approach generates several extreme estimates
relative to others, which makes its variance even much larger than
that of the HT estimator.
\begin{figure}[ht]
  \begin{center}
    \leavevmode
    \includegraphics[scale=0.5, angle=0]{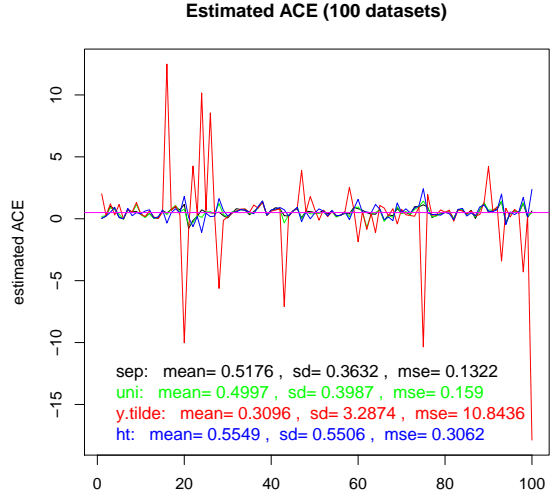}
    \caption{Precision of the estimated $\ace$ based on: (1) specified
      model for $\e_\idle(Y_{i}\;|\;X_{i}, T_{i})$; (2) specified
      models for $\e_\idle(Y_{i}\;|\;X_{i})$ separately for both
      groups; (3) Horvitz-Thompson estimator; (4) regression of
      $\tilde{Y_{i}}$ on $X_{i}$.}
    \label{fig:dreff}
  \end{center}
\end{figure}

\subsubsection{Known response regression model}
Suppose that $\e_\idle(Y_{i}\;|\;X_{i}, T_{i} = 1)$ and
$\e_\idle(Y_{i}\;|\;X_{i}, T_{i} = 0)$ are both known but not the
PM. Then the AIPW estimator can be constructed as:
$$\acehat_{AIPW} = n^{-1}\{\sum_{i=1}^n [\frac{T_{i}}{g(X_{i})} - \frac{1-T_{i}}{1-g(X_{i})}](Y_{i}-m(X_{i})\},$$
where
$$m(X_{i}) = (1-g(X_{i}))\e(Y_{i}\;|\;X_{i}, T_{i} = 1)+g(X_{i})\e(Y_{i}\;|\;X_{i}, T_{i} = 0),$$
and $g(X_{i})$ is an arbitrary function of $X_{i}$.

So $\acehat_{AIPW}$ is unbiased and its variance is computed as
follows.
\begin{eqnarray*}
  \var(\acehat_{AIPW}) &=& \var\{n^{-1}[\sum_{i=1}^n (\frac{T_{i}}{g(X_{i})} - \frac{1-T_{i}}{1-g(X_{i})})(Y_{i}-m(X_{i})]\}\\
  &=& n^{-2}\var\{\sum_{i=1}^n (\frac{T_{i}}{g(X_{i})} - \frac{1-T_{i}}{1-g(X_{i})})(Y_{i}-[1-g(X_{i}))\mu_{1i}+g(X_{i})\mu_{0i}]\}\\
  &=& n^{-2}\var\{\sum_{i=1}^n (\mu_{1i}-\mu_{0i}) + \frac{T_{i}}{g(X_{i})}(Y_{i} - \mu_{1i}) - \frac{1-T_{i}}{1-g(X_{i})}(Y_{i} - \mu_{0i})\}\\
  &=& n^{-2}\var\{\sum_{i=1}^n (\mu_{1i}-\mu_{0i})\}\\
  &\;\;\;\;\; +& n^{-2}\e\{\var[\sum_{i=1}^n \frac{T_{i}}{g(X_{i})}(Y_{i} - \mu_{1i}) - \frac{1-T_{i}}{1-g(X_{i})}(Y_{i} - \mu_{0i})\;|\;X_{i}]\}\\
  &>& n^{-2}\var\{\sum_{i=1}^n (\mu_{1i}-\mu_{0i})\} = \var(\acehat_{RRM}).
\end{eqnarray*}
Hence, we conclude that, for each individual, if the
conditional expectations of the response given $X_{i}$ for both groups
are known or correctly specified, then $\acehat_{AIPW}$ will be
less precise than the estimated ACE from the response regressions.

\subsubsection{Discussion}
If the PM is known, then the variance of $\acehat_{AIPW}$ is minimised
when $m(X_{i})$ is specified as in (\ref{eq:mimin2}) -- where separate
specification of $m_1(X_{i})$ and $m_0(X_{i})$ are not
necessary. Rubin and van de Laan \cite{dr:08} has introduced a
weighted response serving as an alternative, but we have shown, by
simulations, that it could result in large variance of the estimated
$\ace$ and possibly larger than the HT estimator. In the case that the
RRM is correctly specified, \ie, $m_1(X_{i}) =
\e_\idle(Y_{i}\;|\;X_{i}, T_{i} = 1)$ and $m_0(X_{i}) =
\e_\idle(Y_{i}\;|\;X_{i}, T_{i} = 0)$, then these two models rather
than the AIPW estimator should be used to estimate $\ace$ for higher
precision of the estimator.

\section{Summary}
In this chapter, we have addressed statistical causal inference using
Dawid's decision-theoretic framework within which assumptions are, in
principle, testable.  Throughout, the concept of sufficient covariate
plays a crucial role. We have investigated propensity analysis in a
simple normal linear model, as well as in logistic model,
theoretically and by simulation.  Adding weight to previous evidence
\cite{jh:98, jr:92, kh:03, wcw:04, ss:07}, our results show that
propensity analysis does little in improving estimation of the
treatment causal effect, either unbiasedness or precision.  However,
as part of the augmented inverse probability weighted estimator that
is doubly robust, correct propensity score model helps provide
unbiased average causal effect.

\section*{Appendix}
\addcontentsline{toc}{section}{Appendix}
\textbf{R code of simulations and data analysis}

\begin{verbatim}

################################################################
Figure 5: Linear regression (homoscedasticity)              
----------------------------------------------------------------
1. Y on X;                                                    
2. Y on population linear discriminant / propensity variable LD; 
3. Y on sample linear discriminant / propensity variable LD*;     
4. Y on population linear predictor LP.                          
################################################################

##  set parameters

p <- 2
delta <- 0.5
phi <- 1
n <- 20

alpha <- matrix(c(1,0), nrow=1)
sigma <- diag(1, nrow=p)
b <- matrix(c(0,1), nrow=p)


##  create a function to compute ACE from four linear regressions
 
ps <- function(r) {

  #  data for T, X and Y from the specified linear normal model 

  set.seed(r)
  .Random.seed
  t <- rbinom(n, 1, 0.5)

  require(MASS)
  m <- rep(0, p)
  ex <- mvrnorm(n, mu=m, Sigma=sigma)
  x <- t%*%alpha + ex
 
  ey <- rnorm(n, mean=0, sd=sqrt(phi))
  y <- t*delta + x%*%b + ey

  #  calculate the true and sample linear discriminants
 
  ld.true <- x%*%solve(sigma)%*%t(alpha)
  pred <- x%*%b

  d1 <- data.frame(x, t)
  c <- coef(lda(t~.,d1))
  ld <- x%*%c

  #  extract estimated average causal effect (ACE) 
  #  from the four linear regressions

  dhat.pred <- coef(summary(lm(y~1+t+pred)))[2] 
  dhat.x <- coef(summary(lm(y~t+x)))[2]
  dhat.ld <- coef(summary(lm(y~t+ld)))[2]
  dhat.ld.true <- coef(summary(lm(y~t+ld.true)))[2]

  return(c(dhat.x, dhat.ld, dhat.ld.true, dhat.pred)) 
}


##  estimate ACE from 200 simulated datasets
##  compute mean, standard deviation and mean square error of ACE
 
g <- rep(0, 4)
for (r in 31:230) {
  g <- rbind(g, ps(r))
}
g <- g[-1,]

d.mean <- 0
d.sd <- 0
mse <- 0

for (i in 1:4) {
  d.mean[i] <- round(mean(g[,i]),4)
  d.sd[i] <- round(sd(g[,i]),4)
  mse[i] <- round((d.sd[i])^2+(d.mean[i]-delta)^2, 4)
}


##  generate Figure 5

par(mfcol=c(2,2), oma=c(1.5,0,1.5,0), las=1)
main=c("M0:  Y on (T, X=(X1, X2)')", "M3:  Y on (T, LD*)", 
    "M1: Y on (T, LD=X1)", "M2:  Y on (T, LP=X2)")

for (i in 1:4){
  hist(g[,i], br=seq(-2.5, 2.5, 0.5), xlim=c(-2.5, 2.5), ylim=c(0,80), 
       main=main[i], col.lab="blue", xlab="", ylab="",col="magenta")
  legend(-2.5,85, c(paste("mean = ",d.mean[i]), paste("sd = ",d.sd[i]), 
       paste("mse = ",mse[i])), cex=0.85, bty="n")
}
mtext(side=3, cex=1.2, line=-1.1, outer=T, col="blue", 
    text="Linear regression (homoscedasticity) [200 datasets]")

dev.copy(postscript,"lrpvpdecmbook.ps", horiz=TRUE, paper="a4")
dev.off()
 
            
###########################################################################
Linear regression and subclassification (heteroscedasticity)              
---------------------------------------------------------------------------

Figure 6:
1. Regression on population linear predictor LP;                                                    
2. Regression on population linear discriminant LD; 
3. Regression on population quadratic discriminant / propensity variable QD;     
4. Subclassification on QD.  

Figure 7: 
1. Regression on sample linear predictor LP*;                                                    
2. Regression on sample linear discriminant LD*; 
3. Regression on sample quadratic discriminant / propensity variable QD*;     
4. Subclassification on QD*.                                            
###########################################################################           
 
##  set parameters
          
p <- 20
d <- 0
delta <- 0.5
phi <- 1
n <- 500

a <- matrix(rep(0,p), nrow=1)
alpha <- matrix(c(0.5,rep(0,p-1)), nrow=1)
sigma1 <- diag(1, nrow=p)
sigma0 <- diag(c(rep(0.8, 10), rep(1.3, 10)), nrow=p)
b <- matrix(c(0, 1, rep(0,p-2)), nrow=p)


##  create a function to compute ACE from eight approaches

ps <- function(r) {

  #  data for T, X and Y from the specified linear normal model 

  set.seed(r)
  .Random.seed
  pi <- 0.5
  t <- rbinom(n, 1, pi)
  n0 <- 0
  
  for (i in 1:n) {
    if (t[i]==0)
    n0 <- n0+1
  }
  
  t <- sort(t, decreasing=FALSE)
  mu1 <- a+alpha
  mu0 <- a

  require(MASS)
  m <- rep(0, p)
  ex0 <- mvrnorm(n0, mu=m, Sigma=sigma0)
  ex1 <- mvrnorm((n-n0), mu=m, Sigma=sigma1)

  a <- matrix(rep(a, n), nrow=n, byrow=TRUE)
  x0 <- a[(1:n0),] + t[1:n0]%*%alpha + ex0
  x1 <- a[(n0+1):n,] + t[(n0+1):n]%*%alpha + ex1
  x <- rbind(x0, x1)

  ey <- rnorm(n, mean=0, sd=sqrt(phi))
  d <- rep(d, n)
  y <- d + t*delta + x%*%b + ey

  #  calculate linear discrimant, quadratic discrimant, for population
  #  and for sample, extract estimated ACE from linear regressions 

  ld <- x%*%solve(pi*sigma1+pi*sigma0)%*%t(alpha)
  d1 <- data.frame(x, t)
  c <- coef(lda(t~.,d1))
  ld.s <- x%*%c

  z1 <- x%*%(solve(sigma1)%*%t(mu1) - solve(sigma0)%*%t(mu0))
  z2 <- 0
  for (j in 1:n){
    z2[j] <-  - 1/2*matrix(x[j,], nrow=1)%*%(solve(sigma1)
      - solve(sigma0))%*%t(matrix(x[j,], nrow=1))
  }
  qd <- z1+z2

  dhat.x2 <- coef(summary(lm(y~1+t+x[,2])))[2] 
  dhat.ld <- coef(summary(lm(y~1+t+ld)))[2]
  dhat.qd <- coef(summary(lm(y~1+t+qd)))[2]

  mn <- aggregate(d1, list(t=t), FUN=mean)
  m0 <- as.matrix(mn[1, 2:(p+1)])
  m1 <- as.matrix(mn[2, 2:(p+1)])
  v0 <- var(x0)
  v1 <- var(x1)

  c1 <- solve(v1)%*%t(m1)-solve(v0)%*%t(m0)
  z1.s <- x%*%c1
  c2 <- solve(v1)-solve(v0)
  z2.s <- 0
  for (i in 1:n){
    z2.s[i] <- -1/2*matrix(x[i,], nrow=1)%*%c2%*%t(matrix(x[i,], nrow=1))
  }
  qd.s <- z1.s+z2.s

  dhat.x <- coef(summary(lm(y~1+t+x)))[2]
  dhat.ld.s <- coef(summary(lm(y~1+t+ld.s)))[2]
  dhat.qd.s <- coef(summary(lm(y~1+t+qd.s)))[2]

  #  extract estimated ACE from subclassification 

  d2 <- data.frame(cbind(qd, qd.s, y, t))

  tm1 <- vector("list", 2)
  tm0 <- vector("list", 2)
  te.qd <- 0

  for (k in 1:2) {
    d3 <- d2[, c(k,3,4)]
    d3 <- split(d3[order(d3[,1]), ], rep(1:5, each=100))
   
    tm <- vector("list", 5)
    for (j in 1:5) { 
      tm[[j]] <-   aggregate(d3[[j]], list(Stratum=d3[[j]]$t), FUN=mean)
      tm1[[k]][j] <- tm[[j]][2,3]
      tm0[[k]][j] <- tm[[j]][1,3]
    }
    te.qd[k] <- sum(tm1[[k]] - tm0[[k]])/5   
  }

#  return estimated ACE from the eight approaches

  return(c(dhat.x2, te.qd[1], dhat.ld, dhat.qd, 
    dhat.x, te.qd[2], dhat.ld.s, dhat.qd.s))
}


##  estimate ACE from 200 simulated datasets
##  compute mean, standard deviation and mean square error of ACE

g <- rep(0, 8)
for (r in 31:230) {
   g <- rbind(g, ps(r))
}
g <- g[-1,]

d.mean <- 0
d.sd <- 0
d.mse <- 0

for (i in 1:8) {
  d.mean[i] <- round(mean(g[,i]),4)
  d.sd[i] <- round(sd(g[,i]),4)
  d.mse[i] <- round((d.sd[i])^2+(d.mean[i]-delta)^2, 4)
}


##  generate Figure 6

par(mfcol=c(2,2), oma=c(1.5,0,1.5,0), las=1)
main=c("Regression on LP=X2","Subclassification on QD",
    "Regression on  LD=5/9X1","Regression on QD")
for (i in 1:4){
  hist(g[,i], br=seq(-0.1, 1.1, 0.1), xlim=c(-0.1, 1.1), ylim=c(0,80), 
      main=main[i], col.lab="blue", xlab="", , ylab="", col="magenta")
  legend(-0.2,85, c(paste("mean = ",d.mean[i]), paste("sd = ",d.sd[i]), 
      paste("mse = ",d.mse[i])), cex=0.85, bty="n")
}
mtext(side=3, cex=1.2, line=-1.1, outer=T, col="blue", 
    text="Linear regression and subclassification 
    (heteroscedasticity) [200 datasets]")

dev.copy(postscript,"pslrsubtruebook.ps", horiz=TRUE, paper="a4")
dev.off()


##  generate Figure 7
main=c("Regression on X","Subclassification on QD*",
    "Regression on LD*",  "Regression on QD*")
for (i in 1:4){
  hist(g[,i+4], br=seq(-0.1, 1.1, 0.1), xlim=c(-0.1,1.1), ylim=c(0,80), 
      main=main[i], col.lab="blue", xlab="", ylab="", col="magenta")
  legend(-0.2,85, c(paste("mean = ",d.mean[i+4]), paste("sd = ",d.sd[i+4]), 
      paste("mse = ",d.mse[i+4])), cex=0.85, bty="n")
}
mtext(side=3, cex=1.2, line=-1.1, outer=T, col="blue", 
    text="Linear regression and subclassification 
    (heteroscedasticity, sample) [200 datasets]")

dev.copy(postscript,"pslrsubbook.ps", horiz=TRUE, paper="a4")
dev.off()
           


######################################################################
Figure 9 and Table 1: Propensity analysis of custodial sanctions study              
----------------------------------------------------------------------
1. Y on X;                                                    
2. Y on population linear discriminant / propensity variable LD; 
3. Y on sample linear discriminant / propensity variable LD*;     
4. Y on population linear predictor LP.                          
######################################################################

##  read data, imputation by bootstrapping for missing data

dAll = read.csv(file="pre_impute_data.csv", as.is=T, sep=',', header=T)

set.seed(100)
.Random.seed
library(mi)
data.imp <- random.imp(dAll)


## estimate propensity score by logistic regression

glm.ps<-glm(Sentenced_to_prison~
                Age_at_1st_yuvenile_incarceration_y +
                N_prior_adult_convictions +
                Type_of_defense_counsel +
                Guilty_plea_with_negotiated_disposition +
                N_jail_sentences_gr_90days +
                N_juvenile_incarcerations +
                Monthly_income_level +
                Total_counts_convicted_for_current_sentence +
                Conviction_offense_type +
                Recent_release_from_incarceration_m +
                N_prior_adult_StateFederal_prison_terms +
                Offender_race +
                Offender_released_during_proceed +
                Separated_or_divorced_at_time_of_sentence +
                Living_situation_at_time_of_offence +
                Status_at_time_of_offense +
                Any_victims_female, 
                data = data.imp, family=binomial)
                        
summary(glm.ps)             
eps <- predict(glm.ps, data = data.imp[, -1], type='response')
d.eps <- data.frame(data.imp, Est.ps = eps)


## Figure 9: densities of estimated propensity score (prison vs. probation) 

library(ggplot2)

d.plot <- data.frame(Prison = as.factor(data.imp$Sentenced_to_prison), 
    Est.ps = eps)
pdf("ps.dens.book.pdf")
ggplot(d.plot, aes(x=Est.ps, fill=Prison)) + geom_density(alpha=0.25) +
    scale_x_continuous(name="Estimated propensity score") +
    scale_y_continuous(name="Density")
dev.off()


## logistic regression of the outcome on all 17 variables

glm.y.allx<-glm(Recidivism~ 
                Sentenced_to_prison +                                                                 
                Age_at_1st_yuvenile_incarceration_y +                  
                N_prior_adult_convictions +
                Type_of_defense_counsel +
                Guilty_plea_with_negotiated_disposition +
                N_jail_sentences_gr_90days +
                N_juvenile_incarcerations +
                Monthly_income_level +
                Total_counts_convicted_for_current_sentence +
                Conviction_offense_type +
                Recent_release_from_incarceration_m +
                N_prior_adult_StateFederal_prison_terms +
                Offender_race +
                Offender_released_during_proceed +
                Separated_or_divorced_at_time_of_sentence +
                Living_situation_at_time_of_offence +
                Status_at_time_of_offense +
                Any_victims_female, 
                data = d.eps, family=binomial)

summary(glm.y.allx)             


## logistic regression of the outcome on the estimated propensity score

glm.y.eps<-glm(Recidivism ~ Sentenced_to_prison + Est.ps, 
  data = d.eps, family=binomial)
summary(glm.y.eps)


\end{verbatim}

%
%
%


\end{document}